\documentclass[preprint]{elsarticle}
\usepackage [english]{babel}
\usepackage[utf8]{inputenc}
\usepackage[T1]{fontenc}
\usepackage{caption}
\usepackage{booktabs}
\usepackage{graphicx}

\usepackage{url}
\usepackage{algorithm}
\usepackage{algpseudocode}
\usepackage{parskip}
\usepackage{makecell}
\usepackage{xfrac}
\usepackage{color}
\usepackage{xifthen}% provides \isempty test
\usepackage{hyperref}
\usepackage{bm}
\usepackage{soul}
\usepackage{xcolor}
\usepackage{tikz}
\usepackage{pgfplots}
\pgfplotsset{compat=1.7}
\usepackage{amsthm}
\usepackage{amssymb}
\usepackage{amsmath}
\usepackage{amsfonts}
\usepackage{amssymb}
\usepackage{booktabs}
\definecolor{refkey}{gray}{.65}
\definecolor{labelkey}{gray}{0}

\makeatletter
\def\thm@space@setup{%
  \thm@preskip=\parskip \thm@postskip=0pt
}
\makeatother

\DeclareFontFamily{U}{mathx}{\hyphenchar\font45}
\DeclareFontShape{U}{mathx}{m}{n}{
	<5> <6> <7> <8> <9> <10>
	<10.95> <12> <14.4> <17.28> <20.74> <24.88>
	mathx10
	}{}
\DeclareSymbolFont{mathx}{U}{mathx}{m}{n}
\DeclareMathSymbol{\bigtimes}{1}{mathx}{"91}

\DeclareMathOperator*{\Kprod}{\tikz[anchor=base, baseline]\node[circle,inner sep=1pt,draw] {\small K};}%\mathrlap{\circle}K}%
\DeclareMathOperator*{\Kprods}{\tikz[anchor=base, baseline]\node[circle,inner sep=.5pt,draw] {\tiny K};}%\mathrlap{\circle}K}%

\journal{CMAME}

\usepackage{color}

\definecolor{darkgreen}{rgb}{0,0.6,0}
\DeclareMathOperator{\SPAN}{span}

\newcommand{\lspan}{\textnormal{span}}
\newcommand{\supp}{\textnormal{supp}\,}
\newcommand{\csupp}{\textnormal{csupp}\,}

\newcommand{\Quad}{\mathbb X}
\newcommand{\Cost}{\mathfrak C}

\newcommand{\Ah}{\mathtt A}
\newcommand{\ah}{\mathtt a}

\newcommand{\Pprod}[2]{#1^{\scalebox{.35}[.7]{\,$\le$\,}#2}}
\newcommand{\Pfact}[2]{#1^{\scalebox{.35}[1]{$=$}#2}} 
\newcommand{\Mprod}[1]{\Pprod{M}{#1}}
\newcommand{\Nprod}[1]{\Pprod{N}{#1}}
\newcommand{\Qprod}[1]{\Pprod{\Quad}{#1}}

\newcommand{\phif}[1]{\Pfact\phi{#1}}
\newcommand{\psif}[1]{\Pfact\psi{#1}}
\newcommand\Block[1]{\Pfact B{#1}}
\newcommand\Blocka[1]{\Pfact{\textbf w}{#1}}

\DeclareMathOperator{\nnz}{nnz}
\DeclareMathOperator{\NNZ}{NZ}

\newcommand{\PBox}{\mathcal D}
\newcommand{\PPart}{\mathcal P}

\newtheorem{theorem}{Theorem}
\newtheorem{lemma}[theorem]{Lemma}
\newtheorem{corollary}[theorem]{Corollary}
\newtheorem{proposition}[theorem]{Proposition}
\newtheorem{note}[theorem]{Note}

\theoremstyle{definition}\newtheorem{example}{Example}
\theoremstyle{remark}\newtheorem{remark}{Remark}

\begin{document}

\begin{frontmatter}

\title{Sum factorization techniques in Isogeometric Analysis}

\author[UNIPV]{Andrea Bressan}
\ead{andrea.bressan@unipv.it}

\author[RICAM]{Stefan Takacs}
\ead{stefan.takacs@ricam.oeaw.ac.at}

\address[UNIPV]{Dipartimento di Matematica, Università di Pavia, Pavia, Italia}
\address[RICAM]{Johann Radon Institute for Computational and Applied Mathematics (RICAM),\\
Austrian Academy of Sciences, Linz, Austria}

\begin{abstract}
	The fast assembling of stiffness and mass matrices is a key issue in isogeometric
	analysis, particularly if the spline degree is increased. 
	We present two algorithms based on the idea of sum factorization, one for matrix
	assembling and one for matrix-free methods, and study the behavior of
	their computational complexity in terms of the spline
	order $p$.
	Opposed to the standard approach, these algorithms do not apply the idea element-wise,
	but globally or on macro-elements. If this approach is applied to Gauss quadrature,
	the computational complexity grows as 
	$p^{d+2}$ instead of $p^{2d+1}$ as previously achieved.
\end{abstract}


\end{frontmatter}

\section{Introduction}

Isogeometric Analysis,~\cite{Hughes:2005}, was proposed more than a decade ago
as a new approach for the discretization of partial
differential equations (PDEs) and has gained much interest since then.
Spline spaces, such as spaces spanned by tensor-product B-splines or NURBS,
are typically used for geometry representation in standard CAD systems. In
Isogeometric Analysis, one uses such spaces for the geometry representation
of the computational domain and as space of ansatz functions for the solution of the PDE.

The fast assembling of stiffness and mass matrices is a key issue in Isogeometric
Analysis, particularly if the spline degree is increased. If the assembling is done
in a naive way, and $p^d$ quadrature points are used per element,
the computational complexity of assembling the
mass or the stiffness matrix has order $N p^{3d}$, where $N$
is the number of unknowns, $p$ is the spline order (degree~+~1), and $d$ is the domain dimension.
In recent years, much effort was set on proposing faster assembling schemes; we
want to name particularly the methods of \emph{sum factorization},
\emph{low rank assembling}, and \emph{weighted quadrature}.
Several papers \cite{quad1,quad2,quad3,quad4,quad5} proposed  quadrature
schemes that improve over Gauss quadrature by reducing the number of quadrature nodes by a
$p$-independent factor. Since these methods preserve the tensor-product structure, they
can be combined with the presented techniques in a straight-forward manner.

\emph{Sum factorization} was originally proposed for spectral methods and later applied to high-order finite element
methods \cite{orszag,Melenk:2001,Ainsworth:2011}.
Antolín, Buffa, Calabrò, Martinelli, and Sangalli~\cite{Antolin:2015} have carried
over this approach to the case of Isogeometric Analysis and have shown that the
computational complexity of assembling a standard mass of stiffness matrix can be
reduced to order $ N p^{2d+1}$.

Authors from the same group have then further reduced the computational complexity
by \emph{weighted quadrature}, cf. the publication by Calabrò, Sangalli, and
Tani~\cite{Calabro:2017} and the related publication~\cite{bartovn2017efficient}. Here, the idea is to reduce the number of quadrature
points by setting up appropriately adjusted quadrature rules. This allows
to assemble a standard mass or stiffness matrix with a computational complexity
of order $N p^{d+1}$. This, however, comes with the cost that the resulting
matrix is non-symmetric and a careful analysis is necessary to show that the
overall discretization satisfies the expected error bounds.

\emph{Low rank assembling} is based on a completely different idea.
It is observed that for practical problems the mass
and the stiffness matrices have a small tensor rank and are well approximated by a sum of few simple terms.
Mantzaflaris, Jüttler, Khoromskij, and Langer~\cite{Mantzaflaris:2017}
have discussed how to set up an assembling algorithm based on this
approach. Hofreither~\cite{Hofreither:2017} and Georgieva and
Hofreither~\cite{GeorgievaHofreither:2016} have
shown that by rewriting
the problem accordingly, standard adaptive cross approximation algorithms
can be used as black-box methods to determine the mass or stiffness matrices.
For these approaches, the overall computational complexity for assembling the system matrix is of order $NR p^d$, 
where $R$ is the (unknown) tensor-rank of the resulting matrix. 
The complexity a matrix-free application is of order $N R p$.

In this article, we present a \emph{global} variant of sum factorization.
We give two algorithms, one for assembling and one for matrix-free methods.
We derive a complexity analysis that allows to estimate the computational
complexity for many situations.
We observe that the computational complexity of sum-factorization is reduced
if the idea is not applied on a per-element basis but globally.
The cost of assembling the system matrix using Gauss quadrature by using global
sum-factorization is of order $N p^{d+2}$, unlike
per-element sum-factorization that has order $N p^{2d+1}$. 
Still, this algorithm yields the
same matrix, up to machine precision, which would be obtained with straight-forward assembling.
We will then see that the computational complexity is still preserved
for a localized approach, which allows the extension to Hierarchical B-splines (HB splines).
We observe that, like for the low rank assembling, the
formation of the matrix is the most expensive part and that matrix-free approaches
yield smaller computational costs than approaches based on assembling. Besides
Gauss quadrature, our abstract analysis covers any other tensor product quadrature rules
and allows to recover the results for weighted quadrature.

This article is organized as follows. In Section~\ref{sec:2}, we state the
abstract formulation of the problem and give examples of bilinear forms
falling into the class. We introduce the global assembling procedure in
Section~\ref{sec:3}. In Section~\ref{sec:4}, we discuss the localization of
sum factorization and its applications. Then, in Section~\ref{sec:5}, we give an
algebraic description of the proposed algorithm. In Section~\ref{sec:6}, we explain why the evaluation of the basis functions and the geometry function is non trivial. In Section~\ref{sec:7} we give numerical experiments. Finally, the conclusions are given in Section~\ref{sec:8}.

\section{Bilinear forms and their discretization}\label{sec:2}

The goal of this paper is to discuss fast algorithms for tensor-product discretizations of bilinear forms. We consider bilinear forms
$a:H^{r}([0,1]^{d})\times H^{s}([0,1]^{d})\to\mathbb R$ of the form
\begin{equation}\label{eq:high-order}
a(u,v)
:= \sum_{\theta\in \Theta_r}\sum_{\eta\in \Theta_s} 
\int_{[0,1]^d}\mathcal F_{\theta,\eta}(\textbf{x})\;
 			\partial^{\theta} u(\textbf{x})\;
 			\partial^{\eta}v (\textbf{x})\;
 			\mathrm{d} \textbf{x} ,
\end{equation}
where $\Theta_r\subset  \mathbb N^d$ is the set of multi-indices $\theta=(\theta_1,\ldots,\theta_d)$ corresponding to partial derivatives of order up to $r$. Moreover, we consider block-systems where each
block has such a structure.

A Petrov-Galerkin discretization is performed by choosing sub-spaces
$$U_h \subset H^{r}([0,1]^{d}) \quad \mbox{and} \quad  V_h \subset H^{s}([0,1]^{d})$$
and by restricting the bilinear form $a(\cdot,\cdot)$ to these spaces. The special case $U_h=V_h$,
which is known as Galerkin discretization, is certainly also covered by this formulation.
A matrix representation $A$ is obtained by choosing a basis $\Phi=(\phi_n)_{n=1}^N$ for $U_h$ and a basis $\Psi=(\psi_m)_{m=1}^M$ for $V_h$ and setting
\begin{equation}\label{eq:ass:0}
	A = [
		a(\phi_n,\psi_m)
	]_{n=1,\ldots,N}^{m=1,\ldots,M}.
\end{equation}

An exact evaluation of the bilinear form $a(\cdot,\cdot)$ is typically not feasible.
In practice, the bilinear form is approximated using a quadrature rule.
Let $\Quad\subset \mathbb R^d$ be the set of quadrature points and $\omega:\Quad\to\mathbb R$ be a weight function. 
Then, $a(u,v)$ is approximated by
\begin{equation}\label{eq:quadrature:1}
 \ah(u,v) := \sum_{\theta\in \Theta_r}\sum_{\eta\in \Theta_s} \sum_{\textbf{x} \in \Quad} 
								\omega(\textbf{x})\, \mathcal{F}_{\theta,\eta} (\textbf{x}) \,\partial^\theta u(\textbf{x})\, \partial^\eta v(\textbf{x}),
\end{equation}
and the matrix $A$ is approximated by $\Ah=[
	\ah(\phi_n,\psi_m)
]_{n=1,\ldots,N}^{m=1,\ldots,M}$.

While the use of tensor-product discretizations seems restrictive, the following two examples illustrate that the standard isogeometric method employing B-splines or NURBS fits into the considered class of problems.
\begin{example}[Convection diffusion equation]\label{example:second-order}
	We consider a standard single-patch isogeometric discretization, so we assume that the computational
	domain $\Omega$ is parametrized by a diffeomorphism
	\[
			\textbf{G}:\widehat{\Omega}:=[0,1]^d \rightarrow \Omega := \textbf{G}(\widehat{\Omega}) \subset \mathbb{R}^d.
	\]
	We assume that a source function $f\in L^2(\Omega)$ and coefficient-functions
	$A\in L^\infty(\Omega, \mathbb{R}^{d\times d})$, $\textbf{b}\in L^{\infty}(\Omega, \mathbb{R}^d)$ and $c\in L^{\infty}(\Omega)$
	are given. The boundary value problem reads as follows. Find $u\in H^1(\Omega)$ such that
	\begin{equation}\nonumber
			- \nabla \cdot (A \nabla u) + \textbf{b} \cdot \nabla u + c u= f \quad\mbox{ in }\Omega, \qquad
			\frac{\partial u}{\partial n} = 0 \quad\mbox{ on }\partial \Omega.
	\end{equation}
	The variational formulation reads as follows. Find $u\in H^1(\Omega)$ such that
	\begin{equation*}
				a(u,v):=
					\Big(
							\left[
									\begin{array}{cc}
											c & 0\\
											\textbf{b} & A
									\end{array}
							\right]  \left[
									\begin{array}{c}
											u\\
											\nabla  u
									\end{array}
							\right],\left[
									\begin{array}{c}
											v\\
											\nabla v
									\end{array}
							\right]
					\Big)_{L^2(\Omega)}
					= (f, v)_{L^2(\Omega)} \  \mbox{for all }v  \in H^1(\Omega).
	\end{equation*}
	A standard isogeometric discretization is set up on the parameter domain, i.e., we first define the spline
	space
	\begin{align*}
	\widehat{V}_h & :=\lspan\{\widehat \varphi_1,\ldots,\widehat \varphi_{N}\}
							:= \bigotimes_{\delta=1}^d \lspan \{\Pfact{\widehat \varphi}{\delta}_{1},\ldots,\Pfact{\widehat \varphi}{\delta}_{{N_{\delta}}}\}
						\subset H^1(\widehat{\Omega}),
	\end{align*}
	where $\Pfact{\widehat \varphi}{\delta}_{{n}}$ are the standard B-spline basis
	functions as given by the Cox-de~Boor formula.
	Then, the ansatz functions are transferred to the physical domain using the pull-back principle
	\[
			V_h := \widehat{V}_h \circ \textbf{G}^{-1}, \quad \varphi_{n}:= \widehat{\varphi}_{n} \circ \textbf{G}^{-1}.
	\]

	For the computation of the stiffness matrix $\Ah$, we transfer the functions of interest to the parameter domain and obtain
	\begin{align}
			&A  = [a(\varphi_{n}, \varphi_{m})]_{n=1\ldots,N}^{m=1,\ldots,N} \nonumber\\
			&= \Big[
	\Big(
		\underbrace{| J_{\textbf{G}}|
		\left[
		\begin{array}{cc}
		I &\\
		& J_\textbf{G}^{-\top}
		\end{array}
	\right]
		\left[
		\begin{array}{cc}
		\widehat{c} & 0\\
		\widehat{\textbf{b}} & \widehat A
		\end{array}
	\right]
	\left[
		\begin{array}{cc}
		I &\\
		& J_\textbf{G}^{-1}
		\end{array}
	\right] }_{\displaystyle \mathcal{F}:=}
	\left[
		\begin{array}{c}
		1\\
		\nabla
		\end{array}
	\right] \widehat \varphi_n,
	\left[
		\begin{array}{c}
		1\\
		\nabla
		\end{array}
	\right] \widehat \varphi_m
	\Big)_{L^2(\widehat \Omega)} \Big]_n^m\nonumber\\
				& = \sum_{\theta,\eta\in\Theta_1} \Big[
				 \int_{\widehat \Omega}
						\mathcal{F}_{\theta,\eta} (\textbf{x})\;
						\partial^{\theta} \widehat{\varphi}_n (\textbf{x})\;
						\partial^{\eta} \widehat{\varphi}_m (\textbf{x})\;
						\mathrm{d} \textbf{x}
				 \Big]_{n=1,\ldots,N}^{m=1,\ldots,N}
	\label{eq:cdr:1}
	\end{align}
		where $ J_{\textbf{G}}$ is the Jacobi matrix of the
		geometry function and $| J_{\textbf{G}}|$ the absolute value of its determinant.
		We see that the variational formulation of the convection diffusion equations
		belongs to the class of considered bilinear forms. If
		we replace the integral by a quadrature formula, we obtain an 
		approximation $\ah(u,b)$ as in~\eqref{eq:quadrature:1}.
\end{example}

As mentioned above, we can also consider bilinear forms that represent
systems of differential equations.

\begin{example}[Stokes equation]\label{example:stokes}
As a second example, consider we the Stokes system. For ease of notation, we
assume $\Omega=[0,1]^2$. In variational formulation, the Stokes equations
read as follows. Find $u \in [H^1(\Omega)]^2$ and $p \in L^2(\Omega)$ such that
\begin{equation}\nonumber
\left \{\begin{aligned}
(\nabla u,\nabla v)_{L^2(\Omega)} + ( p, \nabla \cdot v)_{L^2(\Omega)}  &= (f,v)_{L^2(\Omega)}&\quad
\mbox{for all}
\quad
v \in [H^1(\Omega)]^2
\\
(\nabla\cdot u,q)_{L^2(\Omega)}\phantom{  + (\nabla u,\nabla w)_{L^2(\Omega)}  }&= 0&\quad
\mbox{for all}
\quad
q \in [L^2(\Omega)]^2.\,
\end{aligned}
\right .
\end{equation}
In a Galerkin discretization, the variational problem is restricted to sub-spaces
$$
	V_h^{(1)}  \subset H^1(\Omega),\qquad
	V_h^{(2)}  \subset H^1(\Omega)\qquad \mbox{and} \qquad
	Q_h  \subset L^2(\Omega)
$$
that satisfy the discrete inf-sup stability. In~\cite{Bressan:2013},
inf-sup stability has been shown for the following two methods.

For the \emph{isogeometric Taylor-Hood method}, the pressure space $Q_h$ is the space of
B-splines of some order $p$ with maximum smoothness. The velocity spaces $V_h^{(1)}$
and $V_h^{(2)}$ are the spaces B-splines of order $p+1$ and
reduced smoothness, $p-2$, on the same grid.

For the \emph{isogeometric sub-grid method}, the pressure space $Q_h$ is again the space of
B-splines of some order $p$ with maximum smoothness. The velocity spaces $V_h^{(1)}$
and $V_h^{(2)}$ are the spaces of B-splines of order $p+1$ with
maximum smoothness on a grid obtained from the pressure grid by one dyadic refinement step.

In both cases, a Galerkin discretization of such a system yields a block-matrix
$$
		A  = \left[
		\begin{array}{ccc}
		A_{11} &        &  B_1^\top\\
		       & A_{22} &  B_2^\top\\
		B_1    & B_2 & 0
		\end{array}
	\right] ,
$$
where each of the blocks $A_{\delta\delta}$ is obtained by a Galerkin discretization of the
scalar valued bilinear form $a_{\delta\delta}(\cdot,\cdot)$, given by
$$
		a_{\delta\delta}(u,v) :=
				\int_{\Omega} 
				\partial^{(1,0)} u(\textbf{x})\; \partial^{(1,0)} v(\textbf{x})\; \mathrm{d}\textbf{x}
 			+
				\int_{\Omega} 
				\partial^{(0,1)} u(\textbf{x})\; \partial^{(0,1)} v (\textbf{x})\;\mathrm{d}\textbf{x}
$$
and each of the blocks $B_{\delta}$ is obtained by a Petrov-Galerkin discretization of the
scalar valued bilinear form $b_{\delta}(\cdot,\cdot)$, given by
$$
		b_{1}(u,q) :=
				 \int_{\Omega} 
				\partial^{(1,0)} u(\textbf{x})\;  q(\textbf{x})\; \mathrm{d}\textbf{x} 
		\quad\mbox{ and }\quad
		b_{2}(u,q) :=
				 \int_{\Omega} 
				\partial^{(0,1)} u(\textbf{x})\;  q (\textbf{x})\; \mathrm{d}\textbf{x}
				.
$$
Here, the idea of sum factorization can be employed for each of the blocks of
the overall system. Note that -- although we have considered a standard Galerkin
discretization to the whole Stokes system -- the discretization of the blocks
$B_1$ and $B_2$ has the form of a Petrov-Galerkin type discretization.
\end{example}

\section{Sum factorization}\label{sec:3}

\subsection{Prerequisites and formulation of the algorithm}

In the last section, we have posed bilinear forms of the form~\eqref{eq:high-order},
which yield problems of the form~\eqref{eq:quadrature:1} if the integral is approximated
by quadrature formulas. First of all we rewrite \eqref{eq:quadrature:1} as a sum of simpler bilinear forms
\begin{equation}\label{eq:quadrature:2}
		\ah(u,v) = \sum_{\theta\in \Theta_r} \sum_{\eta \in \Theta_s}
			\ah_{\theta,\eta}(\partial^{\theta} u, \partial^{\eta} v),
\end{equation}
where
$$
	\ah_{\theta,\eta}(\phi,\psi)
	= \sum_{\textbf{x} \in \Quad} 
				\omega(\textbf{x})\;
				\mathcal{F}_{\theta,\eta} (\textbf{x}) \;
				\phi(\textbf{x})\;
				\psi(\textbf{x}),
		\quad
		\phi \in \partial^{\theta} U
		\quad\mbox{and}\quad
		\psi \in \partial^{\eta} V.
$$
Observe that we incorporate the derivatives into the spaces, not into
the bilinear forms. This is rather unusual from the point of view of
numerical analysis of partial differential equations, but is reasonable to
consider this formulation when discussing the assembling procedure of
stiffness matrices.

The idea is to assemble a matrix $\Ah_{\theta,\eta}$ for each of the bilinear forms $\ah_{\theta,\eta}(\cdot,\cdot)$. Then the matrix $\Ah$ is computed by 
\begin{equation}\label{eq:A-sum}
	\Ah=\sum_{\theta\in\Theta_r}\sum_{\eta\in\Theta_s} \Ah_{\theta,\eta}.
\end{equation}
All the matrices $\Ah_{\theta,\eta}$ are assembled by the same algorithm, but changing the coefficients $\mathcal F_{\theta,\eta}$ and the generating systems $\partial^\theta\Phi:=(\partial^\theta \phi_n)_{n=1}^N$ of $\partial^\theta U_h$ and $\partial^\theta\Psi:=(\partial^\theta \psi_m)_{m=1}^M$ of $\partial^\eta V_h$ that depend on $\theta$ and $\eta$.
Since $\partial^\theta\Phi$ and $\partial^\eta\Psi$
can contain linearly dependent functions for $\theta\ne 0$,
we call them generating sets rather than bases.

Since the algorithm is independent of $\theta$ and $\eta$, we drop the sub-indices $\theta$ and $\eta$ in its description and in its complexity analysis. This means that in the remainder we are considering the assembling of a matrix $\Ah$ corresponding to
\begin{equation}\label{eq:ass:1}
	\ah(\phi_n,\psi_m)
	:= \sum_{\textbf{x} \in \Quad} 
				\omega(\textbf{x})\;
				\mathcal{F} (\textbf{x}) \;
				\phi_n(\textbf{x})\;
				\psi_m(\textbf{x}),
\end{equation}
where the functions $\phi_n$ and $\psi_m$ are taken from the sets $\Phi$ and $\Psi$ that are not necessarily bases.

As already mentioned in the last section, the sum factorization algorithm requires that the discretization and the quadrature have a tensor-product structure.
This means that 
\begin{itemize}
\item \emph{Tensor-product discretization:} for each direction $\delta=1,\ldots,d$, there exist
sets of univariate functions $\Phi_\delta$ and $\Psi_\delta$ such that
\begin{equation}\label{eq:prelim-tp}
\begin{aligned}
\Phi&=\Phi_1\otimes\ldots\otimes\Phi_d, \qquad \Phi_\delta=(\phif\delta_{n})_{n=1}^{N_\delta},\\
\Psi&=\Psi_1\otimes\ldots\otimes\Psi_d,  \qquad \Psi_\delta=(\psif\delta_{m})_{m=1}^{M_\delta}.
\end{aligned}
\end{equation}
\end{itemize}

The meaning of \eqref{eq:prelim-tp} is that the $d$-variate functions in $\Phi$
are the products of the univariate functions in $\Phi_1,\ldots,\Phi_d$ and analogously the functions in $\Psi$.
By convention, we assume the lexicographic ordering.
For a rigorous definition, let $\pi:\{1,\dots, N \} \rightarrow \{1,\dots,N_{1}\}\times \dots\times \{1,\dots,N_{d}\}$
be a lexicographic ordering, i.e., the bijection defined by
$$
\pi:=(\pi_1,\ldots,\pi_d),\quad
\pi_\delta(n):=\left\lfloor \frac{(n-1) \bmod \Nprod{\delta} }{\Nprod{\delta-1} }\right\rfloor+1, \quad
\Nprod{\delta} := \prod_{i=1}^{\delta} N_i .
$$
We define $\sigma :\{1,\dots, M \} \rightarrow \{1,\dots,M_{1}\}\times \dots\times \{1,\dots,M_{d}\}$
and $\Mprod{\delta}$
analogously. Using these orderings, we denote the generating
functions for $\Phi$ and $\Psi$ as follows:
\begin{align*}
 \quad\phi_n(x_1,\dots,x_d)&= \phif1_{\pi_1(n)}(x_1)\cdots \phif d_{\pi_d(n)}(x_d), &&  n=1,\dots,N,\\
 \psi_m(x_1,\ldots,x_d) &= \psif1_{\sigma_1(m)}(x_1)\cdots \psif d_{\sigma_d(m)}(x_d), && m=1,\dots,M.
\end{align*}
\begin{note}
	When considering problems of the form \eqref{eq:quadrature:2}
	then the tensor-product structure~\eqref{eq:prelim-tp} of all the generating systems $\partial^\theta\Phi$ follows from that  of $\Phi$:
	$$
	\partial^{\theta}\Phi =(\tfrac{\partial^{\theta_1}}{\partial x^{\theta_1}}\cdots\tfrac{\partial^{\theta_d}}{\partial x^{\theta_d}} \phi_n)_{n=1}^{N}= \partial^{\theta_1}\Phi_1 \otimes \cdots\otimes\, \partial^{\theta_d}\Phi_d,
	$$
	where $\partial^{\theta_\delta}\Phi_\delta:=(\tfrac{\partial^{\theta_\delta}}{\partial x^{\theta_\delta}} \phif\delta_{n})_{n=1}^{N_\delta}$.
\end{note}

We also require a
\begin{itemize}
\item \emph{Tensor-product quadrature:} for each direction $\delta=1,\ldots,d$, there exists
a set of quadrature points $\Quad_\delta\subset\mathbb R$ and a weight function $\omega_\delta:\Quad_\delta\rightarrow\mathbb R$ such that
\begin{equation}
\label{eq:prelim-quad-tp}
\begin{aligned}
\Quad&=\Quad_1\times\ldots\times\Quad_d,\\
\omega(\textbf x)&=\omega_1(x_1)\ldots \omega_d(x_d).
\end{aligned}
\end{equation}
\end{itemize}
Analogous to $\Nprod{\delta}$, it is convenient to define
$
		\Qprod{\delta} :=\Quad_1\times\ldots\times\Quad_\delta 
$.

The method of sum  factorization is based on the observation that
expanding each term in~\eqref{eq:ass:1} with respect to its components  yields
\begin{equation}\nonumber
\begin{aligned}
	&\ah(\phi_n,\psi_m)
	=   \sum_{x_1 \in  \Quad_1} \cdots \sum_{x_d \in  \Quad_d}
								\prod_{\delta=1}^d\omega_\delta(x_\delta) \, \prod_{\delta=1}^d \phif\delta_{\pi_\delta(n)}  (x_\delta)\, \prod_{\delta=1}^d \psif\delta_{ \sigma_\delta(m)}  (x_\delta)\, \mathcal{F} (\textbf{x})\\
	&\ = \sum_{x_d \in  \Quad_d}\phif d_{\pi_d(n)} (x_d)\,
						\psif d_{\sigma_d(m)}  (x_d)\, \omega_d(x_d)\\
	&\hspace{1cm}\underbrace{
	\sum_{x_1 \in  \Quad_1} \cdots \hspace{-2.5ex}\sum_{x_{d-1} \in  \Quad_{d-1}}
		\prod_{\delta=1}^{d-1}\omega_\delta(x_\delta)\,
		\prod_{\delta=1}^{d-1}\phif\delta_{ \pi_\delta(n)} (x_\delta)\,
		\prod_{\delta=1}^{d-1}\psif\delta_{\sigma_\delta(m)} (x_\delta)\,
		\mathcal{F} (\textbf x)}_{\displaystyle = :\Pprod{\ah}{d-1}_{x_d}(\Pprod\phi{d-1}_n,\Pprod\psi{d-1}_m) },
 \end{aligned}
\end{equation}
that is
\begin{equation}\label{eq:recursive-step1}
\ah(\phi_n,\psi_m)=\sum_{x_d \in  \Quad_d}\phif d_{\pi_d(n)} (x_d)\,
					\psif d_{\sigma_d(m)}  (x_d)\, \omega_d(x_d)\,
					\Pprod\ah{d-1}_{x_d}(\Pprod\phi{d-1}_n,\Pprod\psi{d-1}_m) ,
\end{equation}
where
$$
		\Pprod\phi{\delta}_n  :=      \prod_{i=1}^\delta  \phif{i}_{\pi_i(n)}	\qquad \mbox{and}\qquad
		\Pprod\psi{\delta}_m  :=     \prod_{i=1}^\delta  \phif{i}_{\sigma_i(m)}	.
$$
The term  $\Pprod{\ah}{d-1}_{x_d}(\Pprod\phi{d-1}_n,\Pprod\psi{d-1}_m)$ is independent of the $d$-th components of $\Phi,\Psi$, $\omega$ and it appears for many $(\phi_n,\psi_m)$ pairs.
The advantage of \eqref{eq:recursive-step1} over \eqref{eq:ass:1} is that it shows that $\Pprod{\ah}{d-1}_{x_d}(\Pprod\phi{d-1}_n,\Pprod\psi{d-1}_m)$ can be computed once and used many times.
By rewriting $\Pprod{\ah}{d-1}_{x_d}(\Pprod\phi{d-1}_n,\Pprod\psi{d-1}_m)$ as
\begin{equation}\label{eq:last-dir-sum}
\Pprod{\ah}{d-1}_{x_d}(\Pprod\phi{d-1}_m,\Pprod\psi{d-1}_n)=\sum_{\textbf x\in\Qprod{d-1}}\Pprod{\omega}{d-1}(\textbf x) \Pprod{\phi_{n}}{d-1}(\textbf x)\Pprod{\psi_{m}}{d-1}(\textbf x) \mathcal{F} (\textbf x,x_d),
\end{equation}
we see that, analogously to \eqref{eq:ass:1}, it is an approximation with the quadrature $(\Qprod{d-1}, \Pprod\omega{d-1})$ of the bilinear form on $L^2([0,1]^{d-1})$ defined by $$\Pprod a{d-1}_{x_d}(\phi,\psi)=\int_{[0,1]^{d-1}} \phi(\textbf x)\psi(\textbf x) \mathcal F(\textbf x, x_d)\, d\textbf x.$$

Let $\Pprod{\Ah}{d-1}_{x_d}$ be the matrix representing the Petrov-Galerkin restriction of $\Pprod a{d-1}_{x_d}$ to $\SPAN\Pprod{\Phi}{d-1}\times\SPAN \Pprod{\Psi}{d-1}$, i.e.,
 $$
\Pprod{\Ah}{d-1}_{x_d}=[\Pprod{\ah}{d-1}_{x_d}(\Pprod\phi{d-1}_n,\Pprod\psi{d-1}_m)]^{m=1,\dots,\Mprod{d-1}}_{n=1,\dots,\Nprod{d-1}}.
$$

According to \eqref{eq:recursive-step1}, the entries of $\Ah$ are linear combinations of the entries in $\Pprod{\Ah}{d-1}_{x_d}$ for different $x_d\in\Quad_d$, but for the same $\Pprod\phi{d-1}_n$ and $\Pprod\psi{d-1}_m$.
This suggests the decomposition of $\Ah$ into blocks.
The $d$-th components of $\phi_n$ and $\psi_m$, i.e. $\phif d_{\pi_d (n)}$ and $\psif d_{\sigma_d (m)}$, identify a block of size $\Mprod{d-1}\times \Nprod{d-1}$. Let $\Block d_{i,j}$ denote these blocks as in
\begin{equation}\label{eq:matrix:parts}
\Pprod\Ah{d} =\begin{pmatrix} \Block d_{1,1} & \dots & \Block d_{1,N_d}\\
\vdots & &\vdots\\
 \Block d_{M_d,1} & \dots & \Block d_{M_d,N_d}
 \end{pmatrix}.
\end{equation}
The non-zero coefficients of $\Pprod{\Ah}{d-1}_{x_d}$ 
only depend on $\Pprod\Phi{d-1}$, $\Pprod\Psi{d-1}$
and $\Pprod\Quad{d-1}$ and are independent of $x_d$.
They determine a common maximum for the sparsity patterns of the blocks $\Block \delta_{m_\delta,n_\delta}$ that the following recursive assembling procedure uses for the sum in \eqref{eq:last-dir-sum}.

\begin{algorithm}[H]
	\caption{Recursive sum factorization}
	\label{alg-rec-sm}
	\begin{algorithmic}[1]
	\Procedure{Assemble}{$[ \Quad_\delta]_{\delta=1}^d,[\omega_\delta]_{\delta=1}^d, [\Phi_\delta]_{\delta=1}^d,[\Psi_\delta]_{\delta=1}^d,[\mathcal{F} (\textbf{x}) ]_{x\in\Quad}$}
	\If {d=0}
		\State \textbf{return} $[\mathcal{F} (\textbf{x})]_{x\in\Quad }$
	\EndIf
	\State $\Pprod\Ah d  \gets0$  \Comment{start with null matrix}
	\ForAll{$x_d\in \Quad_d$}     \Comment{sum all $\Pprod\Ah {d-1}_{x_d}$}
		\State 
		$\displaystyle\begin{aligned}\Pprod\Ah{d-1}_{x_d}  \gets  \textsc{Assemble} ( &
				[ \Quad_\delta]_{\delta=1}^{d-1},
				[\omega_\delta]_{\delta=1}^{d-1},\\
				&[\Phi_\delta]_{\delta=1}^{d-1},
				[\Psi_\delta]_{\delta=1}^{d-1},
				[\mathcal{F}(\textbf x, x_d)]_{\textbf x\in\Qprod{d-1 }} )\end{aligned}$
		\ForAll{ $n_d$ \textbf{with} $\phif d_{n_d}(x_d)\ne 0$}\label{alg_line:nd}
		\ForAll{ $m_d$ \textbf{with} $\psif d_{m_d}(x_d)\ne 0$}\label{alg_line:md}
		\State \mbox{$\Block d_{m_d,n_d}\gets \Block d_{m_d,n_d}+\omega_d(x_d)\phif d_{n_d} (x_d)\psif d_{m_d} (x_d) \Pprod\Ah{d-1}_{x_d}$} \Comment{ cf. \eqref{eq:matrix:parts}}
			\label{alg_line:assemb}
		\EndFor
 		\EndFor
	\EndFor
	\State \textbf{return} $\Ah=\Pprod\Ah d$
	\EndProcedure
	\end{algorithmic}
\end{algorithm}

\subsection{Complexity analysis}

For the complexity analysis, we need to bound the number
of non-vanishing matrix entries.
The coefficient $(n,m)$ of $\Ah$ can be non-zero only if there is
at least one quadrature node $\textbf x$ such that $\phi_n(\textbf x)\ne0$ and
$\psi_m(\textbf x) \ne 0$. (The other direction is not true.) 
We estimate the number from above by considering the convex hull of the
support (instead of the support itself). So, we know that
the coefficient $(n,m)$ of $\Ah$ can be non-zero only if
$(n,m)\in \NNZ(\Phi,\Psi,\Quad)$, where
$$
	\NNZ(\Phi,\Psi,\Quad):=\{(n,m): \Quad\cap\csupp\phi_n\cap\csupp\psi_m \ne \emptyset \},
$$
and $\csupp \phi$ is the convex hull of the support of $\phi$, i.e.,
$$
	\csupp \phi := \{ \alpha \textbf x + (1-\alpha) \textbf y  : \phi(\textbf x) \ne 0,\,
	\phi(\textbf y) \ne 0,\, \alpha\in[0,1] \}.
$$

\begin{note}
	This estimate of the sparsity pattern carries over to
	derivatives.
	Provided that $\partial^\theta\Phi$ and $\partial^\eta\Psi$ are defined on all quadrature
	points, we have
	$$
		\NNZ(\partial^\theta\Phi,\partial^\eta\Psi,\Quad)\subseteq\NNZ(\Phi,\Psi,\Quad).
	$$
	This means that for problems of the form \eqref{eq:quadrature:2} we can use the same sparsity pattern for all matrices $\Ah_{\theta,\eta}$ by storing the coefficients at row $m$ and column $n$ with $(n,m)\in\NNZ(\Phi,\Psi,\Quad)$.
	This simplifies the computation of the sum in \eqref{eq:A-sum}.
\end{note}

The number of non-zero entries of the matrix $\Ah$ (and later the complexity result)
will be expressed in terms of
the number of quadrature points ($\#\Quad_\delta$),
the number of trial functions ($N_\delta$),
the number of test functions ($M_\delta$) and the
\emph{overlap parameters} $p_{\delta}$ and $q_{\delta}$, defined by
\begin{align}\label{def:delta}
\begin{aligned}
p_{\delta}=\max_{x\in \Quad_\delta} \#\{\phi\in\Phi_\delta: x\in\csupp\phi \},\\
q_{\delta}=\max_{x\in \Quad_\delta} \#\{\psi\in\Psi_\delta: x\in\csupp\psi \},
\end{aligned}
\end{align}
where $\# T$ denotes the number of elements of the set $T$. The interpretation of
$p_\delta$ is as follows. At any quadrature node, no more than $p_\delta$ trial functions
are active (in the sense that the quadrature node belongs to the convex hull of the support).
The interpretation of $q_\delta$ is analogous.
For the convenience of the reader, we will often express the main results also in terms of $$p:=\max\{p_1,\ldots,p_d, q_1,\dots,q_d\}.$$

\begin{note}
	When $\Phi$ is a B-spline or NURBS basis of degree $\nu$ (order $\nu+1$), the condition \eqref{def:delta} holds with $p_\delta = \nu+1$. The same applies to $\Psi$ and $q_\delta$.
\end{note}

	First of all we compute a univariate upper bound on the number of non-zero coefficients.
	For $\phi\in\Phi_\delta$, $\csupp\phi$ is an interval which intersects $\Quad_\delta$ in a finite number of points.
Let $x_\phi$ be the leftmost point of $\csupp\phi\cap\Quad_\delta$, i.e.
\[
x_\phi=\min \big(\csupp\phi\cap\Quad_\delta\big).
\]
Define $x_\psi$ analogously for $\psi\in\Psi_\delta$.

\begin{lemma}\label{lemma:1D-NNZ}
For $\delta=1,\dots,d$
\[
	\NNZ(\Phi_\delta,\Psi_\delta,\Quad_\delta)=
	 \{(n,m): x_{\phi_n} \in \csupp \psi_m \}\cup 
	\{(n,m): x_{\psi_m} \in \csupp \phi_n \}.
\]
\end{lemma}

\begin{figure}[th]
	\def\support#1#2#3{\fill[opacity=.2] (#1,#2) rectangle +(3,.5) node[midway,opacity=1] {#3}; \draw[opacity=.5,thin] (#1,-.2)-- +(0,#2+.7) (#1+3,-.2)-- +(0,#2+.7); }
	\def\quad{\foreach \t in {-0.6,-0.2,0.4,0.8,1.2,1.6,2.2,2.6,3.2,3.8,4.2,4.6,5.0,5.4}
	\fill (\t,0) circle (1pt);}
	\centering\begin{tikzpicture}
		\begin{scope}
			\support{0}{0}{$\csupp \phi$}
			\support{1.8}{.525}{$\csupp \psi$}
			\quad
			\draw[black] (0.4,0) circle (2pt) node[below] {$x_\phi$};
			\draw[black] (2.2,0) circle (2pt) node[below] {$x_\psi$};
			\draw  (-2,0) node  {\begin{minipage}{2cm}
				$x_{\psi}\in\csupp\phi$
				\\$x_{\phi}\not\in\csupp\psi$
			\end{minipage} };
		\end{scope}

		\begin{scope}[yshift=2cm]
			\support{1.8}{0}{$\csupp \phi$}
			\support{0}{.525}{$\csupp \psi$}
			\quad
			\draw[black] (0.4,0) circle (2pt) node[below] {$x_\psi$};
			\draw[black] (2.2,0) circle (2pt) node[below] {$x_\phi$};
			\draw  (-2,0) node  {\begin{minipage}{2cm}
				$x_{\psi}\not\in\csupp\phi$
				\\$x_{\phi}\in\csupp\psi$
			\end{minipage} };
		\end{scope}

		\begin{scope}[yshift=4cm]
			\support{0.2}{0}{$\csupp \phi$}
			\support{0}{.525}{$\csupp \psi$}
			\quad
			\draw[black] (0.4,0) circle (2pt) node[below] {$x_\phi=x_\psi$};
			\draw  (-2,0) node  {\begin{minipage}{2cm}
				$x_{\psi}\in\csupp\phi$
				\\$x_{\phi}\in\csupp\psi$
			\end{minipage} };
		\end{scope}

	\begin{scope}[yshift=-2cm]
		\support{2.9}{0}{$\csupp \phi$}
		\support{-.5}{.525}{$\csupp \psi$}
		\quad
		\draw[black] (-0.2,0) circle (2pt) node[below] {$x_\psi$};
		\draw[black] (3.2,0) circle (2pt) node[below] {$x_\phi$};
		\draw  (-2,0) node  {\begin{minipage}{2cm}
			$x_{\psi}\not\in\csupp\phi$
			\\$x_{\phi}\not\in\csupp\psi$
		\end{minipage} };
	\end{scope}
	\end{tikzpicture}
	\caption{The relative position of  $\csupp \psi$ and $\csupp \phi$ determines if
	$x_\phi$ belongs to $\csupp \psi$ and if $x_\psi$ to $\csupp \phi$. The dots are the quadrature points in $\Quad$.}\label{fig:lem:nnz-1D}
	\end{figure}
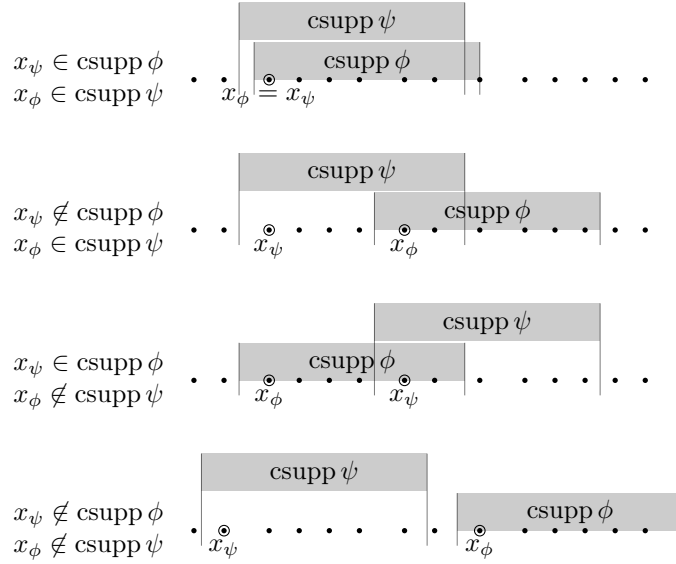
\begin{proof}
	From Figure~\ref{fig:lem:nnz-1D} we notice that if $(n,m)\in \NNZ(\Phi_\delta,\Psi_\delta,\Quad_\delta)$, i.e., $\Quad_\delta\cap\csupp\phi_n\cap\csupp\psi_m \ne \emptyset$, then $x_{\phi_n}\in\csupp \psi_m$ or $x_{\psi_m}\in\csupp \phi_n$.
\end{proof}	

\begin{corollary}\label{lem:nnz-1D}
	For all $\delta=1,\ldots,d$, we have
	$$
	\#\NNZ(\Phi_\delta,\Psi_\delta,\Quad_\delta)\le p_\delta M_\delta + q_\delta N_\delta.
	$$
\end{corollary}
\begin{proof}
	From \eqref{def:delta} we deduce that for any fixed $n$, $\#\{(n,m): x_{\phi_n} \in \csupp \psi_m \}\le q_\delta$.
	Thus the first set in Lemma~\ref{lemma:1D-NNZ} contains at most $q_\delta N_\delta$ elements.
	Similarly, for $m$ fixed, $\#\{(n,m): x_{\psi_m} \in \csupp \phi_m \}\le p_\delta$ and the second set contains at most $p_\delta M_\delta$ elements.
\end{proof}

Observe that Corollary~\ref{lem:nnz-1D} gives an upper bound for the size of the
sparsity pattern. For the case of B-splines or NURBS, the following Corollary gives a
precise statement on the size of $\#\NNZ(\Phi_\delta,\Psi_\delta,\Quad_\delta)$.

\begin{corollary}\label{lem:nnz-B-spline}
	Let $\Phi_\delta$ and $\Psi_\delta$ be univariate B-spline or NURBS bases of order $p_\delta$ and $q_\delta$, respectively, defined over open knot vectors on the interval $[a,b]$.
	Further, let $\Xi$ be the union of knot values in the knot vectors of $\Phi_\delta$ and $\Psi_\delta$ and
	let $\mu_\Phi(\xi)$ and $\mu_\Phi(\xi)$ be the multiplicity of $u$ in the knot vector of $\Phi_\delta$ or $\Psi_\delta$, respectively.
	If for all $\xi_1<\xi_2\in \Xi$ there is some
	$x\in\Quad_\delta\cap(\xi_1,\xi_2)$, i.e., there is at least one
	quadrature point between any two consecutive knots, then  
	\[
		\#\NNZ(\Phi_\delta,\Psi_\delta,\Quad_\delta) 
		 = q_\delta N_\delta + p_\delta M_\delta - \sum_{\xi\in \Xi\setminus\{b\}} \mu_\Phi(\xi)\mu_\Psi(\xi).
	\]
\end{corollary}
\begin{proof}
Let $S_\Phi=\{(n,m): x_{\phi_n} \in \csupp \psi_m \}$ and $S_\Psi=\{(n,m): x_{\psi_m} \in \csupp \phi_n \}$.
Note that \eqref{def:delta} holds with equality for all quadrature nodes $x_{\phi_n}$ and
$x_{\psi_m}$. Consequently, we have
$$
	\# S_\Phi= q_\delta N_\delta,\qquad \# S_\Psi=p_\delta M_\delta.
$$
Since $x_{\phi_n} \in \csupp \psi_m\Rightarrow  x_{\phi_n}\ge x_{\psi_m}$ and $x_{\psi_m} \in \csupp \phi_n\Rightarrow  x_{\psi_m}\ge x_{\phi_n}$ we deduce that 
$$
	S_\Phi\cap S_\Psi=\{(n,m): x_{\phi_n}=x_{\psi_m}\}.
$$
By assumption there is a node of $\Quad_\delta$ between any two knots in $\Xi$, thus $(n,m)\in S_\Phi\cap S_\Psi$ implies that the first knot of $\phi_n$ equals the first knot of $\psi_m$. As each knot $\xi\in\Xi$ is the first knot for $\mu_\Phi(\xi)$ functions in $\Phi_\delta$ and $\mu_\Psi(\xi)$ functions in $\Psi_\delta$ we deduce that $\mu_\Phi(\xi)\mu_\Psi(\xi)$ pairs are in $S_\Phi\cap S_\Psi$ and the thesis follows.
\end{proof}

For the complexity analysis, we assume that the values of~$\mathcal F$, $\phi_n$ and $\psi_m$ are already
available; cf. Section~\ref{sec:6} on evaluating them and the related computational complexity.
The recursive assembling procedure provides a recursive formula for the computational cost for assembling:
$$
\Pprod\Cost d \eqsim \# \Quad_d \big (\Pprod\Cost{d-1} + p_d\, q_d\  \nnz (\Pprod{\Ah}{ d-1}_{x_d})\big ),
$$
where $\Pprod\Cost{d-1}$ is the cost of assembling $\Pprod{\Ah}{ d-1}_{x_d}$ and
$\nnz (\Pprod{\Ah}{ d-1}_{x_d})$ is the number of non-zero entries of $\Pprod{\Ah}{ d-1}_{x_d}$.
To keep the notation tight,
here and in what follows,  $a\lesssim b$ means that there is a constant $c>0$, independent of $p_\delta$, $N_\delta$, $q_\delta$, $M_\delta$ and $\Quad_\delta$, but possibly depending $d$ such that $a \le c\,b$.
Moreover, $a\eqsim b$ means that $a\lesssim b$ and $b \lesssim a$.
The factor $p_d\,q_d$ comes from the number of iterations at line \ref{alg_line:nd} and
\ref{alg_line:md} that are bounded using \eqref{def:delta}.
Letting $\Pprod\Ah{d-2}_{x_{d-1},x_d}$ be the analog of $\Pprod{\Ah}{ d-1}_{x_d}$ and so on till $\Pprod\Ah0_{\textbf x}= \mathcal F(\textbf x)$ and expanding the recursive cost formula yields
\begin{equation}
\label{sum factorization-nnz}
\Pprod\Cost{d}\eqsim \sum_{i=1}^d \Big(\prod_{\delta=i+1}^{d} \# \Quad_\delta\Big)  (p_i\,q_i \# \Quad_i ) \ \nnz (\Pprod\Ah{i-1}_{x_i,\dots,x_d}).
\end{equation}
Lemma~\ref{lem:nnz-1D} and the Kronecker like position of the non-zero coefficients, see \eqref{eq:matrix:parts}, allow us to bound the number of non-zero entries as follows:
\begin{equation}\label{eq:nnz}
\nnz(\Pprod\Ah{i-1}_{x_i,\dots,x_d})\le \prod_{\delta=1}^{i-1} (q_\delta N_\delta + p_\delta M_\delta),
\end{equation}
which yields
\begin{equation}
\label{sum factorization}
\begin{aligned}
\Pprod\Cost{d}&\eqsim \sum_{i=1}^d p_i q_i \Big(\prod_{\delta=i}^{d} \# \Quad_\delta\Big)  \Big(\prod_{\delta=1}^{i-1} (q_\delta N_\delta + p_\delta M_\delta)\Big)
\\
&\le\Big(\sum_{i=1}^d p_i q_i\Big) \#\Quad_d  \prod_{\delta=1}^{d-1} \max\{\#\Quad_\delta,q_\delta N_\delta + p_\delta M_\delta\} .
\end{aligned}
\end{equation}
To compare the above with the standard approaches we assume $M_\delta\lesssim N_\delta$ and consider the following cases
\begin{align}
\label{sum factorization-N}
&\forall \delta,\,\#\Quad_\delta\lesssim  q_\delta N_\delta + p_\delta M_\delta &\Rightarrow&&& \Pprod\Cost{d}\lesssim p^{d+2} \, N ,\\
\label{sum factorization-N2}
&\forall \delta,\,\#\Quad_\delta\lesssim   N_\delta    &\Rightarrow&&& \Pprod\Cost{d}\lesssim p^{d+1} \, N,\\
\label{sum factorization-Q}
&\forall \delta,\,\#\Quad_\delta\gtrsim   q_\delta N_\delta + p_\delta M_\delta &\Rightarrow&&& \Pprod\Cost{d}\lesssim  p^2\,  \#\Quad.
\end{align}

The above formulas explain the costs for isogeometric discretizations using Gauss and weighted quadrature as shown in the following examples.

\begin{example}[Gauss quadrature]\label{example:gauss}
In Isogeometric Analysis (cf. Examples~\ref{example:second-order}
and~\ref{example:stokes}), typically spline spaces are
used for discretization. The breakpoints of these spline spaces introduce a decomposition of the parameter domain into $\prod_{\delta=1}^d K_\delta$ elements with
\[
K_\delta \le N_\delta - p_\delta +1.
\]
On each element the splines are polynomials and are typically integrated using tensor-product Gauss quadrature with $p_1\times\ldots\times p_d$ quadrature points per element.

Using~\eqref{sum factorization-N}, we obtain
\begin{align*}
\#\Quad_\delta=p_\delta K_\delta \le p_\delta N_\delta \quad\mbox{and}\quad  \Pprod\Cost{d}\lesssim p^{d+2}\, N .
\end{align*}
\end{example}
\medskip
\begin{example}[Weighted quadrature]\label{example:weighted-quad}
A way to reduce the computational cost is to reduce the number of quadrature points.
In~\cite{Calabro:2017} it was shown that
this can be done without loosing accuracy using  test-function dependent weights:
$$
\omega(\textbf x, m)=\prod_{\delta=1}^d \omega_\delta( x_\delta, \sigma_\delta(m) ).
$$
The adaptation of Algorithm~\ref{alg-rec-sm} to test-function-dependent quadrature formulas is straightforward. Indeed,
it can equivalently be thought as a quadrature with constant weights $\omega(\textbf x)=1$ and a different test
space $\widetilde\Psi$ whose functions are pre-multiplied by the quadrature weights.
The construction in~\cite{Calabro:2017} uses $\# \Quad_\delta \lesssim N_\delta $ quadrature points per direction and the reported cost $\Pprod\Cost{d}\lesssim  p^{d+2} N$
is explained by~\eqref{sum factorization-N2}.

Note however that this strategy breaks the symmetry
between test and trial functions and consequently the symmetry of the assembled matrix.
\end{example}

\subsection{Matrix-free application}

Iterative solvers typically require only the ability to compute matrix-vector products $\textbf v=\Ah \textbf u$, while a direct access to the entries of $\Ah$ is not required.
Thus the assembling cost can be avoided and the solving time can be improved.
The definition given by~\eqref{eq:ass:0} and~\eqref{eq:ass:1} yields
$$
\Ah \textbf u=  \Big( \sum_{\textbf{x}\in \Quad} \omega(\textbf x)\;
	\psi_m(\textbf x) \;  u(\textbf x)   \Big)_{m=0}^M,
\quad
\mbox{where}
\quad
u(\textbf x)=\sum_{n=1}^N \textbf u_n \phi_n(\textbf x).
$$

Algorithm~\ref{alg-rec-sm} computes $\textbf v$ if $u$ is incorporated in the weight $\mathcal F$ and $\Phi$ contains only the constant $1$.
In this setting the matrix $\Ah$ corresponds to $\textbf v$ and the blocks $\Block d_{i,j}$ to blocks $\Blocka d_i$ containing $\Mprod{d-1}$ coefficients of $\textbf v$:
$$
\textbf v=(\underbrace{\textbf v_1,\dots,\textbf v_{\Mprod{d-1} }}_{\Blocka d_1},\dots,\underbrace{\textbf v_{\Mprod{d}-\Mprod{d-1}+1},\dots,\textbf v_{\Mprod{d} }}_{\Blocka d_{M_d}}).
$$
With these changes and by eliminating the loop at line \ref{alg_line:nd} of Algorithm~\ref{alg-rec-sm} the following matrix-free algorithm is obtained.

\begin{algorithm}[H]
	\caption{Matrix-free sum factorization}
	\label{alg-mf-sm}
	\begin{algorithmic}[1] % The number tells where the line numbering should start
	\Procedure{Apply}{$[ \Quad_\delta]_{\delta=1}^d,[\omega_\delta]_{\delta=1}^d, [\Psi_\delta]_{\delta=1}^d,[\mathcal{F} (\textbf{x})u(\textbf{x}) ]_{x\in\Quad}, $}
	\If {d=0}
		\State \textbf{return} $[\mathcal{F} (\textbf{x})u(\textbf{x})]_{x\in\Quad }$
	\EndIf
	
	\State $\Pprod{\textbf v} d  \gets0$  \Comment{start with null vector}
	\ForAll{$x_d\in \Quad_d$}     \Comment{sum all $\Pprod{\textbf v} {d-1}_{x_d}$}
		\State $\Pprod{\textbf v}{d-1}_{x_d}  \gets $\textsc{Apply}$([ \Quad_\delta]_{\delta=1}^{d-1},[\omega_\delta]_{\delta=1}^{d-1}, [\Psi_\delta]_{\delta=1}^{d-1},[\mathcal{F}(\textbf x, x_d)u(\textbf x, x_d)]_{\textbf x\in\Qprod{d-1 }} )$
		\ForAll{ $m_d$ \textbf{with} $\psif d_{m_d}(x_d)\ne 0$}\label{alg_line:md2}
		\State \mbox{$\Pprod{\textbf{w}}{d}_{m_d}\gets \Pprod{\textbf{w}}{d}_{m_d}+\omega_d(x_d)%\phif d_{n_d} (x_d)
\psif d_{m_d} (x_d) \Pprod{\textbf v}{d-1}_{x_d}$} 
			\label{alg_line:assemb2}
		\EndFor
	\EndFor
	\State \textbf{return} $\textbf v=\Pprod{\textbf v} d$
	\EndProcedure
	\end{algorithmic}
\end{algorithm}

The cost of Algorithm~\ref{alg-mf-sm}, excluding the evaluation of $u(\textbf x)\mathcal F(\textbf x)$, can be deduced from \eqref{sum factorization-nnz} by replacing $\nnz(\Pprod\Ah{i-1}_{x_i,\dots,x_d})$ with $\Mprod{i}$, and $p_\delta$ with one. This yields
\begin{equation}
\label{matrix-less}
\Pprod\Cost{d}_{app}\eqsim \sum_{i=1}^d q_i \Big(\prod_{\delta=i}^{d} \# \Quad_\delta\Big) \ \Mprod {i-1}\le \sum_{i=1}^d q_i \Quad_d \prod_{\delta=1}^{d-1} \max\{\#\Quad_\delta, M_{\delta}\}.
\end{equation}
The overall cost is bounded as follows
\begin{align}
\label{matrix-less-N}
&\forall \delta,\,\#\Quad_\delta\lesssim  p M_\delta   &\Rightarrow&&& \Pprod\Cost{d}_{app}\lesssim   p^{d+1} M,\\
\label{matrix-less-N2}
&\forall \delta,\,\#\Quad_\delta\lesssim  M_\delta    &\Rightarrow&&& \Pprod\Cost{d}_{app}\lesssim   p\, M,\\
\label{matrix-less-Q}
&\forall \delta,\,\#\Quad_\delta\gtrsim  M_\delta &\Rightarrow&&& \Pprod\Cost{d}_{app}\lesssim   p\,  \#\Quad .
\end{align}

\section{Localized sum factorization}\label{sec:4}

We have described sum factorization as a global assembling procedure that requires a tensor-product discretization and the pre-computation of 
%the weights
$\mathcal F$ on all the quadrature nodes in $\Quad$.
This contrasts to the usual FEM assembling procedure which allows for unstructured grids and is performed locally in an element-by-element fashion.
The idea of this section is to apply the methods from Section~\ref{sec:3} locally, i.e., on sub-domains. This accomplishes three distinct goals.

First, it shows that the complexity estimates apply to a broader class of problems, for instance to multi-patch isogeometric discretizations or to discretizations based on locally refined splines such as hierarchical B-splines, T-splines or LR-splines.

Second, by applying sum factorization on sub-domains, even for globally tensor-product discretizations, we greatly reduce the memory requirements. Indeed, the coefficient $\mathcal F$ of the bilinear form  is pre-computed on a fraction of quadrature points at a time.
This also improves data-locality that might lead to better cache utilization.

Finally, the localized approach allows for the same parallelization strategy as in the finite element method. Different sub-domains can be assigned to different execution units and synchronization needs to be employed only for updates to the resulting matrix or vector.

\subsection{Prerequisites and complexity analysis}

Let $\PPart$ be a collection of pairwise disjoint $d$-dimensional boxes $\PBox=\bigtimes_{\delta=1}^d \PBox_\delta$.
We consider a bilinear form $a$ of the type
\begin{equation}\label{eq:high-order-local}
	a(u,v)
	:= 	\sum_{\PBox\in\PPart}\int_{\PBox}\mathcal F (\textbf{x})\;
				 u(\textbf{x})\;
				 v (\textbf{x})\;
				 \mathrm{d} \textbf{x}.
\end{equation}
For each $\PBox$ let $(\Quad_\PBox,\omega_\PBox)$ be a tensor-product quadrature \eqref{eq:prelim-quad-tp}, and let $\ah$ be the approximation of $a$ computed using the quadratures $(\Quad_\PBox,\omega_\PBox)$ as in \eqref{eq:ass:1}.
Let $\Phi=(\phi_n)_{n=1}^N$ and $\Psi=(\psi_m)_{m=1}^M$ be generating systems of functions $\bigcup_{\PBox\in\PPart}\PBox\rightarrow \mathbb R$, and $\Ah$ be the matrix associated to the restriction of $\ah$ to $\SPAN\Phi\times\SPAN\Psi$.

Furthermore, for $\PBox\in \PPart$, let $\Phi_\PBox=(\phi_{\PBox,n})_{n=1}^{N_\PBox}$ and $\Psi_\PBox=(\psi_{\PBox,m})_{m=1}^{M_\PBox}$ be tensor-product generating systems, as in \eqref{eq:prelim-tp}, that contain the restrictions of $\Phi$ and $\Psi$, respectively. This means
that
\begin{equation}\label{eq:local-subset}
\begin{aligned}
&\forall \, n=1,\dots,N,\   &&\phi_n|_{\PBox}\in \Phi_\PBox\cup\{0\},\\
&\forall \, m=1,\dots,M,\  &&\psi_m|_{\PBox}\in \Psi_\PBox\cup\{0\}.
\end{aligned}
\end{equation}
Let $N_{\PBox,\delta}$, $M_{\PBox,\delta}$, $p_{\PBox,\delta}$, $q_{\PBox,\delta}$  be the analogous of $N_\delta$, $M_\delta$, $p_\delta$ and $q_\delta$.
The complexity results are expressed in terms of $N$, $M$, $\#\Quad_{\PBox}$,
 \[
	p=\max_{\PBox\in\PPart}\max_{\delta=1,\dots,d}\max\{p_{\PBox,\delta},q_{\PBox,\delta}\},
\]
and the \emph{repetition ratio} $R$
\begin{align}\label{eq:bounded-repetition}
	R:=\max\Big\{ \frac{\sum_{\PBox\in\PPart} N_\PBox}N, \frac{\sum_{\PBox\in\PPart} M_\PBox}M \Big\}.
\end{align}

\begin{remark}\label{lem:nnz-overlap}
Note that  
\[R \le \max\Big\{ \max_{\phi\in\Phi} \#\{\PBox\in \PPart:  \phi|_{\PBox}\ne 0 \},\, \max_{\psi\in\Psi} \#\{\PBox\in \PPart: \psi|_{\PBox}\ne 0 \}\Big\}.
\]
\end{remark}

Let $\Ah_\PBox$ be the matrix representation of $\ah$ restricted to $\SPAN\Phi_\PBox\times \SPAN\Psi_\PBox$ computed using $(\Quad_\PBox,\omega_\PBox)$.
By \eqref{eq:local-subset}, we have 
\begin{equation}\label{eq:loc-basis-change}
\Ah= \sum_{\PBox\in\PPart} S_{\Psi,\PBox}^\top \Ah_\PBox  S_{\Phi,\PBox},
\end{equation}
where the matrices $S_{\Phi,\PBox}$ and $S_{\Psi,\PBox}$ are selection matrices whose columns and rows contain at most one non zero coefficient.
Consequently, the cost of assembling $\Ah$ decomposes in two parts:
\begin{equation}\label{eq:cost-decomposition}
\Cost = \Cost_{ass} + \Cost_{acc} ,
\end{equation}
where $\Cost_{ass}$ is the cost of assembling the matrices $\Ah_\PBox$ for all $\PBox\in\PPart$ and
$\Cost_{acc}$ is the cost of accumulating the $\Ah_\PBox$ into $\Ah$.

The accumulation cost $\Cost_{acc}$ is determined by the number of coefficient additions in \eqref{eq:loc-basis-change}.
Recalling that $S_{\Phi,\PBox}$, $S_{\Psi,\PBox}$ are selection matrices, using \eqref{eq:nnz} for estimating $\nnz(\Ah_\PBox)$, and assuming $M_{\PBox,\delta}\lesssim N_{\PBox,\delta}$  we have
\[
\begin{aligned}
\Cost_{acc} &\lesssim \sum_{\PBox\in\PPart} \#\NNZ(\Phi_\PBox,\Psi_\PBox,\Quad_\PBox)
\le \sum_{\PBox\in\PPart} \prod_{\delta=1}^d (q_\delta N_{\PBox,\delta}+ p_\delta M_{\PBox,\delta})
\\&\lesssim  p^d \sum_{\PBox\in\PPart}N_\PBox\lesssim p^d \,R\,N.
\end{aligned}
\]

The cost of assembling each $\Ah_{\PBox}$ using the algorithm from Section~\ref{sec:3} is given by \eqref{sum factorization-N}, \eqref{sum factorization-N2} and \eqref{sum factorization-Q}.
Assuming $M_{\PBox,\delta}\lesssim N_{\PBox,\delta}$ as above, summing over $\PBox\in\PPart$, and using \eqref{eq:bounded-repetition} we notice that $\Cost_{ass}\gtrsim\Cost_{acc}$ for all cases. Consequently we have:
\begin{align}
\label{localized-N}
&\forall \PBox,\,\delta,\ \#\Quad_{\PBox,\delta}\lesssim q_{\PBox,\delta} N_{\PBox,\delta} + p_{\PBox,\delta} M_{\PBox,\delta} 
 &\Rightarrow&&&  \Cost \lesssim   R\, p^{d+2}\, N,\\
\label{localized-N2}
&\forall \PBox,\,\delta,\ \#\Quad_{\PBox,\delta}\lesssim N_{\PBox,\delta}       &\Rightarrow&&&  \Cost \lesssim  R\, p^{d+1}\, N,\\
\label{localized-Q}
&\forall \PBox,\,\delta,\ \#\Quad_{\PBox,\delta}\gtrsim q_{\PBox,\delta} N_{\PBox,\delta} + p_{\PBox,\delta} M_{\PBox,\delta}
&\Rightarrow&&&  \Cost \lesssim   p^2  \sum_{\PBox\in\PPart}\#\Quad_{\PBox}.
\end{align}

Similar reasoning can be used for the application of $\Ah$ to a vector $\textbf u$.
We have
\[
\textbf v:=\Ah\textbf u = \sum_{\PBox\in\PPart} S_{\Psi,\PBox}^\top \Ah_{\PBox} S_{\Phi,\PBox} \textbf u.
\]
The cost of applying $\Ah_{\PBox}$, excluding the evaluation of $u(\textbf{x})\mathcal F(\textbf{x})$, is as in \eqref{matrix-less-N}, \eqref{matrix-less-N2} and \eqref{matrix-less-Q}.
Relabeling and accumulating onto $\textbf v$ has cost $\lesssim R\, M$. The cost is dominated by the application of $\Ah_{\PBox}$, and we have 
\begin{align}
\label{matrix-less-N-local}
&\forall \PBox,\,\delta,\,\#\Quad_{\PBox,\delta}\lesssim  p M_{\PBox,\delta}   &\Rightarrow&&& \Cost_{app}\lesssim  R\, p^{d+1}\,  M,\\
\label{matrix-less-N2-local}
&\forall \PBox,\,\delta,\,\#\Quad_{\PBox,\delta}\lesssim  M_{\PBox,\delta}    &\Rightarrow&&& \Cost_{app}\lesssim R\,p\,M,\\
\label{matrix-less-Q-local}
&\forall \PBox,\,\delta,\,\#\Quad_{\PBox,\delta}\gtrsim M_{\PBox,\delta}     &\Rightarrow&&&  \Cost_{app}\lesssim  p\, \sum_{\PBox\in\PPart}\#\Quad_{\PBox}.
\end{align}

We conclude that, in all considered cases, the localized approach has in the worst case a cost that is $R$ times that of the global sum factorization.

\subsection{Localized sum factorization for tensor-product B-spline bases}\label{sec:localized1}

In this subsection we consider the application of localized sum factorization to an isogeometric Galerkin discretization for which $\Phi=\Psi$ is a B-spline basis defined on $[0,1]^d$.
Our aim is to determine how the assembling cost varies depending on the size of the boxes $\PBox$ of $\PPart$.

The breakpoints of $\Phi$ define a partitioning  of $[0,1]^d$ (B\'ezier mesh)
into elements $K$ on which the basis functions coincide with polynomial functions.
For simplicity, we restrict our analysis to partitions $\PPart$ in which the boundaries of the boxes $\PBox$ are aligned with the boundaries of the elements. This means that the interior of no element intersects two boxes.

The restriction to axes-aligned boxes does not destroy the tensor-product structure.
This means that \eqref{eq:prelim-tp} is satisfied by 
		\[
			\Phi_\PBox:=\Psi_\PBox := \{ \phi_n|_{\PBox} : n=1,\ldots,N \} \backslash \{0\}.
		\]

We bound $R$ based on the minimum size of the boxes $\PBox\in\PPart$.

\begin{lemma}\label{remark:macro-size}
If $\Phi=\Psi$ is a tensor-product B-spline basis of order $p_1,\ldots,p_d$ and each $\PBox\in\PPart$ contains at least $s_1\times\ldots\times s_d \ge 1$ elements, we have
\[
	R \le \prod_{\delta=1}^d \left \lceil \frac { s_\delta+p_\delta -1}{s_\delta} \right\rceil.
\]
\end{lemma}

\begin{proof}
The support of $\phif\delta_n$ contains at most $p_\delta$ elements in direction $\delta$.
The number of boxes containing $s_\delta$ elements intersecting $\supp \phif\delta_n$ is maximized if the end of the leftmost box coincides with the end of the leftmost element.
In this case the total number of boxes intersecting $\supp\phif\delta_n$ is
\[
1+\left \lceil \frac { p_\delta -1}{s_\delta} \right\rceil = \left \lceil \frac { s_\delta+p_\delta -1}{s_\delta} \right\rceil.
\]
To conclude we take the product over $\delta=1,\dots,d$ and use Remark~\ref{lem:nnz-overlap}.
\end{proof}

Second, we compute the ratio between $\#\Quad_{\PBox,\delta}$ and $N_{\PBox,\delta}$ when using an element based quadrature.

\begin{lemma}\label{lemma:maximum-smoothness-ratio}
If $\Phi$ is a tensor-product B-spline basis of order $p_1,\ldots,p_d$, $\Quad$ contains $k_1\times\ldots\times k_d$ quadrature points per element and $\PBox$ contains at most  $s_1\times\ldots\times s_d$ elements then
\[
\#\Quad_{\PBox,\delta} \le \frac{s_\delta\,k_\delta}{s_\delta+p_\delta-1}  N_{\PBox,\delta}.
\]
The estimate $\#\Quad_{\PBox,\delta}\le k_\delta N_{\PBox,\delta}$ holds independently of $s_1,\dots,s_d$.
\end{lemma}
\begin{proof}
On the one hand we have
$\#\Quad_{\PBox,\delta}\le s_\delta k_\delta$ and on the other hand
$s_\delta+p_\delta-1\le N_{\PBox,\delta}$.
Finally, $s_\delta /(s_\delta+p_\delta-1)\le 1$.
\end{proof}

Lemmas~\ref{remark:macro-size} and \ref{lemma:maximum-smoothness-ratio} show two competing effects on the cost of localized sum factorization: the upper bound for $R$ increases for small boxes, while the ratio between quadrature points and basis function decreases.
The balance of these two effects is shown by the following two examples.

\begin{example}[Per-element sum factorization with Gauss quadrature]\label{example:element}
Let $\Phi=\Psi$ be a B-spline basis of order $p_1,\dots,p_d$ and $\Quad$ be the Gauss quadrature with $p_1\times\dots\times p_d$ quadrature points per element.

In this example, $\PPart$ is the collection of the elements associated to $\Phi$. The local spaces $\Phi_\PBox,\Psi_\PBox$ and the local quadrature $\Quad_\PBox$ are the restrictions of 
$\Phi$, $\Psi$ and $\Quad$ to $\PBox$ respectively.

Lemma~\ref{remark:macro-size} states $R\le p^d$.
Since $\#\Quad_{\PBox,\delta} = N_{\PBox,\delta}=p_\delta$, then \eqref{localized-N2} and \eqref{matrix-less-N2-local} yield
\[
\Cost \lesssim p^{2d+1}\, N,\quad\text{and}\quad \Cost_{app} \lesssim  p^{d+1}\, N.
\]
The bound for $\Cost$ coincides with the one in \cite{Antolin:2015} where this approach was originally proposed. 
The bound for $\Cost_{app}$ coincides with that for global sum factorization.
\end{example}

\begin{example}[Per macro-element sum factorization with Gauss quadrature]\label{example:macro-element}
Again, let $\Phi=\Psi$ be a B-spline basis of order $p_1,\dots,p_d$ and $\Quad$ be the Gauss quadrature with $p_1\times\dots\times p_d$ quadrature points per element.
In this example, each box  $\PBox\in\PPart$ contains at least $p_1\times\dots\times p_d$ elements.
The local spaces $\Phi_\PBox,\Psi_\PBox$ and the local quadrature $\Quad_\PBox$ are the restrictions of 
$\Phi$, $\Psi$ and $\Quad$ to $\PBox$ respectively.

Lemma~\ref{remark:macro-size} states $R\le 2^d$,
 Lemma~\ref{lemma:maximum-smoothness-ratio} with $s_\delta \gtrsim k_\delta=p_\delta$ gives 
 $\#\Quad_{\PBox}\eqsim p_\delta N_\delta$. Thus, \eqref{localized-N} and \eqref{matrix-less-N-local} yield
\begin{equation}\label{eq:local:costs}
\Cost \lesssim  p^{d+2} \, N,\quad\text{and}\quad \Cost_{app} \lesssim  p^{d+1}\, N,
\end{equation}
i.e., both the assembling and the application cost are, up to a $p$ independent factor, the same as for the global approach.
\end{example}

A more detailed analysis can be performed starting directly
from~\eqref{sum factorization} and \eqref{matrix-less}.
In both formulas there is a product for $\delta=1,\dots,d-1$
of $\max\{\#\Quad_\delta, q_\delta N_\delta+p_\delta M_\delta\}$
and $\max\{\#\Quad_\delta, M_\delta\}$ respectively.
The factor corresponding to $\delta=d$ is always $\#\Quad_d$.

This means that we can exploit a factor of $k_d/p_d$ 
from Lemma~\ref{lemma:maximum-smoothness-ratio} to compensate for $R$ as
long as the boxes are small \emph{only} in direction $d$.
Since the order of directions in the assembling process is arbitrary,
we can choose it in our favor.
This allows us to consider partitions into \emph{narrow macro-elements}, i.e., into boxes containing one element in one direction and $p_\delta$ elements in the other directions. 
A bound on the repetition ratio $R$ follows from the Lemma below.

\begin{lemma}\label{cor:overlap-narrow}
	Let $\Phi=\Psi$ be a tensor-product B-spline basis of order $p_1,\ldots,p_d$ with $d\ge2$
	and assume that each $\PBox\in\PPart$ of size $s_1\times\dots\times s_d$ elements satisfies
	$\#\{ \delta : s_\delta<p_\delta\}\le 1$. Then,
	we have $R \le  2^{d-1} d p$.
\end{lemma}

The Lemma is a direct consequence of Remark~\ref{lem:nnz-overlap} and the following proposition.

\begin{proposition}\label{lemma:narrow-macro-elements}
For all and coverings $\PPart$ of $\bigtimes_{\delta=1}^d[0,p_\delta[$ with disjoint boxes intersecting it and such that for all $\PBox\in\PPart$
\begin{align}
&\PBox=[a_1,b_1[\times\dots\times [a_d,b_d[, \qquad a_i,b_i\in \mathbb Z,\\
&b_\delta-a_\delta= \begin{cases} p_\delta & \delta \ne j_\PBox\\ 1 & \delta = j_\PBox \end{cases}
\end{align}
for some $j_\PBox \in\{ 1,\ldots,d\}$ that depends on $\PBox$, the following inequality holds
$$
\#\PPart \le 2^d+    2^{d-1}\sum_{\delta=1}^d (p_\delta-2)_+,
$$
where $x_+=\max\{x,0\}$.
\end{proposition}

\begin{proof}
We prove the bound by induction on the dimension $d$. For $d=1$ any partition of $[0,p_{1}]$ in segments of length $1$ has exactly $p_{1}$ elements, which yields the upper bound
$2^1+2^0 (p_1-2)_+$. Now consider $d>1$.
For any partition $\PPart$, define 
\begin{align*}
	\PPart_L&=\{\PBox\in\PPart: \PBox\cap[0,1[\times\bigtimes_{\delta=2}^d[0,p_\delta[ \ne \emptyset \},\\
	\PPart_R&=\{\PBox\in\PPart: \PBox\cap[p_1-1,p_1[\times\bigtimes_{\delta=2}^d[0,p_\delta[ \ne \emptyset \},\\
	\PPart_C&=\PPart\setminus(\PPart_L\cup\PPart_R).
\end{align*}
We have
\begin{equation}\label{eq:proof-narrow-1}
	\#\PPart  \le \# \PPart_L +\#\PPart_R + \# \PPart_C.
\end{equation}
The projection $(x_1,\dots,x_d)\to (x_2,\dots,x_d)$ maps the boxes in $\PPart_L$ to disjoint boxes in $\mathbb R^{d-1}$  that cover $\bigtimes_{\delta=2}^d[0,p_\delta[$. The induction hypothesis and the same argument for $\PPart_R$ yield
\begin{equation}\label{eq:proof-narrow-2}
	\# \PPart_L +\#\PPart_R\le 2 \Big(2^{d-1}+2^{d-2} \sum_{\delta=2}^d(p_\delta-2)_+\Big)\le 2^d+2^{d-1} \sum_{\delta=2}^d(p_\delta-2)_+.
\end{equation}
A box $\PBox\in \PPart_C$ can not have length $p_1$ in the first direction. Thus, $\PBox$ has
size $1\times p_2\times \cdots\times p_d$ and is contained in $[i,i+1[\times\mathbb R^{d-1}$ for some $i\in\{1,\dots,p_1-2\}$.
Lemma~\ref{remark:macro-size} states that there are at most $2^{d-1}$ such boxes for each $i=1,\dots,p_1-2$. Thus we have
\begin{equation}\label{eq:proof-narrow-3}
	\# \PPart_C \le 2^{d-1} (p_1-2)_+.
\end{equation}
The combination of \eqref{eq:proof-narrow-1}, \eqref{eq:proof-narrow-2} and \eqref{eq:proof-narrow-3} yields the desired result.
\end{proof}

\begin{remark}
	Assuming $p_1,\ldots,p_d \ge 2$, the bound in Proposition~\ref{lemma:narrow-macro-elements} is sharp.
	Indeed, for $d=1$ it is sharp. Moreover, in this case \eqref{eq:proof-narrow-1} holds with equality. Equality in \eqref{eq:proof-narrow-3} can be realized and an induction argument shows that the same applies to \eqref{eq:proof-narrow-2}.
\end{remark}

\begin{example}[Per narrow-macro-element sum factorization]\label{example:narrow}
	Let $\Phi=\Psi$ be a B-spline basis of order $p,\dots,p$ and $\Quad$ be a quadrature with $k\times\dots\times k$ points per element with $k\le p$.
	
	In this example, each box  $\PBox\in\PPart$ of size $s_1\times \dots \times s_d$ elements satisfies $\#\{\delta: s_\delta< p\}\le 1$.
	The local spaces $\Phi_\PBox,\Psi_\PBox$ and the local quadrature $\Quad_\PBox$ are the restrictions of 
	$\Phi$, $\Psi$ and $\Quad$ to $\PBox$ respectively.
	
	For simplicity, consider first $s_d=1$. Then, $\#\Quad_{\PBox,d} \le k/p N_{\PBox,d}$, while 
	for $\delta<d$ we have from Lemma~\ref{lemma:maximum-smoothness-ratio}
	that $\#\Quad_{\PBox,\delta} \le k N_{\PBox,\delta}$.
	Inserting these bounds into \eqref{sum factorization} and \eqref{matrix-less}
	gives the following bounds on the costs
	\begin{equation}\label{eq:example:narrow}
		\Cost_{\PBox}\lesssim k\, p^d\, N_{\PBox},\quad\text{and}\quad \Cost_{app,\PBox}\lesssim k^{d}\, N_{\PBox}.
	\end{equation}
	Since \eqref{sum factorization} and \eqref{matrix-less} are linear in $\#\Quad_d$ the bound extends to $1<s_d<p$.
	If $s_\delta<p$ for some $\delta\ne d$, then the same bound is achieved by reordering the directions.
	Lemma~\ref{cor:overlap-narrow} yields $R\le 2^{d-1}\,d\,p$. Thus, we
	obtain
	\[
		\Cost\lesssim k\, p^{d+1}\, N,\quad\text{and}\quad \Cost_{app}\lesssim k^{d}\,p\, N.
	\]
	If $k\eqsim p$, as for Gauss quadrature, we have 
	\[
		\Cost\lesssim p^{d+2}\, N,\quad\text{and}\quad \Cost_{app}\lesssim p^{d+1}\, N.
	\]
	If $k\eqsim 1$, as for weighted quadrature, we have
	\[
		\Cost\lesssim p^{d+1}\, N,\quad\text{and}\quad \Cost_{app}\lesssim p\, N.
	\]	
	Concluding, the cost is, up to a degree independent factor, the same as for global assembling.
\end{example}

\subsection{Localized sum factorization for some non-tensor-product bases}\label{sec:localized2}

In this section we show that hierarchical B-splines and multipatch domains 
fit into the localized sum factorization framework presented in this paper.
Other generating systems without global tensor-product structure (like
hierarchical LR splines \cite{Bressan:2015}) could be analyzed
in similar ways. The same holds if adaptivity and multipatch discretizations
are combined.

\begin{example}[HB-splines]\label{example:hb}
	\def\hbasis#1{\mathfrak B_{#1}}
	\def\hsubdomain#1{ \Omega_{#1}}
	\def\ring#1{ \Delta_{#1}}
	
	\def\numSub{L}
	\def\domain{\Omega}
	\def\thus{\Rightarrow}
	\def\hbb{\mathfrak H}
	\def\hbl{\mathfrak L}
	\def\hbas{\gamma}
	
	The Hierarchical B-spline basis (HB) is a basis that breaks the global tensor-product structure and allows for adaptive methods in which only a part of the domain is refined, see~\cite{Buffa:2016}.
	The basis is obtained by \emph{selecting} functions from different tensor-product B-spline bases on different regions of the domain.
	
	Let $\hbasis 1, \ldots, \hbasis L$ be tensor-product B-spline bases of the same
	degree $p$ defined on the domain $\domain$ that generate nested spaces, i.e.,
	\begin{equation}
	  \label{eq:hierarchy}
	  \SPAN\hbasis 1 \subset \cdots \subset \SPAN \hbasis L
	\end{equation}
	and let
	\[
	\domain=:\hsubdomain 1 \supseteq\dots\supseteq\hsubdomain \numSub \supseteq\Omega_{L+1}:=\emptyset
	\]
	be corresponding closed domains.
	For simplicity, we further assume that the $\hbasis \ell$ are a sequence of dyadically
	refined bases and that each of the domains $\hsubdomain \ell$ is a union of elements of
	the B\'ezier mesh of $\hbasis \ell$.
	The \emph{hierarchical basis} (HB-splines) is defined by
	Kraft's \emph{selection} criteria \cite{Kraft:1997}:
	\def\sellev{\mathfrak{S}}
	\begin{equation}
	  \Phi := \bigcup_{\ell=1}^\numSub \sellev_\ell
	  \quad \mbox{with}\quad
	  \sellev_\ell:= \left\{\begin{aligned}\phi \in \hbasis \ell:\,&
	  \supp \phi \subseteq \hsubdomain \ell,\\
	  &\supp \phi \not\subseteq \hsubdomain {\ell+1}\end{aligned}\right\} .
	\end{equation}
	We assume that $\Phi$ is a $\beta$-admissible HB basis, i.e.,
	$\max\hbl(\textbf x) - \min\hbl(\textbf x)+1\le \beta$ for all $\textbf x\in\domain$,
	where
	\[
		\hbl(\textbf x) := \Big\{\ell \ :\ \textbf x \in
		\bigcup_{\phi\in\sellev_\ell} \supp \phi \Big\}
	\]
	is the set of active levels at a point $\textbf x$.
	Moreover, we assume that the partition of $\PPart$ has the form 
	\[
		\PPart=\bigcup_{\ell=0}^L\PPart_\ell	,
	\] 
	where each $\PPart_\ell$ is a partitioning of the \emph{ring} $\ring {\ell}:= \hsubdomain{\ell}\setminus \hsubdomain{\ell+1}$ into macro-elements $\PBox$ as in Example~\ref{example:narrow}, where the size is measured in elements of the B\'ezier mesh of $\hbasis \ell$.
		For this example we consider local bases
		\[
			\Phi_{\PBox,\ell}=\{ \phi \in \hbasis {\ell} \,:\, \supp\phi\cap\PBox \ne \emptyset  \},\qquad \ell \in \hbl(\PBox).
		\]
		Observe that these are tensor product bases, and that their union contains the restriction of $\Phi$ to $\PBox$, so that an analogous of~\eqref{eq:local-subset} holds.
		Analogously to \eqref{eq:loc-basis-change}, we have
		\[
			\Ah=\sum_{\PBox\in\PPart}\;\;\sum_{\ell,\gamma \in\hbl(\PBox)} S_{\PBox,\ell,\gamma}^\top \Ah_{\PBox,\ell,\gamma}S_{\PBox,\ell,\gamma},
		\]
		where $\Ah_{\PBox,\ell,\gamma}$ is the matrix corresponding to the bases $\Phi_{\PBox,\ell}$ and $\Phi_{\PBox,\gamma}$, and the matrices $S_{\PBox,\ell,\gamma}$ are selection matrices.
		Observe that at most $(\#\hbl(\PBox) )^2=\beta^2$ local matrices are assembled
		for each $\PBox$.
		Using $\# \Phi_{\PBox,\ell+1} \ge \# \Phi_{\PBox,\ell}$ and \eqref{eq:example:narrow}, we deduce that 
		\begin{equation}\label{eq:local-cost:hb}
			\Cost_{\PBox}\lesssim k\, p^d\, \beta^2 \# \Phi_{\PBox,\ell},\qquad \Cost_{\PBox,app}\lesssim k^d\, \beta^2 \# \Phi_{\PBox,\ell}
		\end{equation}
	for all $\PBox\in\PPart_\ell$.
	Since
	$\PPart_\ell$ is a partition of $\ring {\ell}$, Lemma~\ref{cor:overlap-narrow} gives
	\begin{equation}\label{eq:R:hb}
		\sum_{\PBox\in\PPart_\ell} \# \Phi_{\PBox,\ell} 
					\le 2^{d-1} d p\  \#\{ \phi \in \hbasis {\ell} \;:\; \supp \phi \cap \ring \ell \ne
					\emptyset \}.
	\end{equation}
	$\beta$-admissibility implies that for $\phi \in \hbasis {\ell}$ we have
	\begin{equation}\nonumber
		\supp \phi \cap \ring \ell \ne
					\emptyset\quad  \Rightarrow\quad \left\{\begin{aligned}&\supp\phi\subseteq \Omega_{\ell-\beta+1},\\
						&\supp\phi\not\subseteq \Omega_{\ell+1}.
					\end{aligned}\right .
	\end{equation}
	Consequently,
	\begin{equation}\label{eq:hb-proof-1}
		\{ \phi \in \hbasis {\ell} \;:\; \supp \phi \cap \ring \ell \ne
					\emptyset \}\subseteq \bigcup_{i=0}^{\beta-1}  \left\{ \begin{aligned}\phi \in \hbasis {\ell} \,:\ & \supp\phi\subseteq\Omega_{\ell-i},\\
					&\supp \phi \not\subseteq \Omega_{\ell-i+1} \end{aligned}\right\}.
	\end{equation}
	Since the bases $\hbasis \ell$ are obtained by dyadic refinement, we have
	\begin{equation}\label{eq:hb-proof-2}
		\#\left\{ \begin{aligned}\phi \in \hbasis {\ell} \,:\ & \supp\phi\subseteq\Omega_{\ell-i},\\
		&\supp \phi \not\subseteq \Omega_{\ell-i+1} \end{aligned}\right\}\le 2^{d i}
		\#\sellev_{\ell-i}.
	\end{equation}
	Equations \eqref{eq:local-cost:hb}, \eqref{eq:R:hb}, \eqref{eq:hb-proof-1} and \eqref{eq:hb-proof-2} yield the overall cost bound
	\[
		\Cost\lesssim k\, p^{d+1}\, 2^{d\beta}\, \beta^2\, N,\qquad \Cost_{app}\lesssim k^d\, p\, 2^{d\beta}\, \beta^2\, N,
	\]
	where $k$ is the number of quadrature points per element and direction.
	We have shown that, compared to Example~\ref{example:narrow}, the costs increase at worst by a $p$-independent factor $2^{d\beta}\, \beta^2$. 
	The extension to bilinear forms involving derivatives is completely straight forward.
	\end{example}
	
	\begin{remark}
	There is another basis of the space of hierarchical B-splines: the Truncated  Hierarchical B-spline
	(THB) basis, cf.~\cite{Giannelli:2012}. Localized sum factorization cannot be directly applied
	to the truncated basis because it does not fulfill \eqref{eq:prelim-tp}.
	Nevertheless, it is possible to express the matrix $\Ah_{\mathrm{THB}}$ corresponding
	to the truncated basis as
	\[
	\Ah_{\mathrm{THB}}= T_\Psi ^\top\ \Ah_{\mathrm{HB}}\ T_\Phi,
	\]
	where $\Ah_{\mathrm{HB}}$ is the system matrix for the hierarchical basis and $T_\Psi $, $T_\Phi$ are the matrices corresponding to change of bases. The number of non zero entries in each column of $T_\Psi$, $T_\Phi$ is bounded by $(p+1)^d 2^{\beta d}$ for $\beta$-admissible dyadically refined hierarchical B-splines.
	Consequently, the cost $\Cost_{T}$ of computing the product by $T_\Psi$ and  $T_\Phi$ is bounded as follows 
	\[
		\Cost_{T}\lesssim  p^{2 d}\, 2^{\beta d}\, (N + M),
	\]
	where we use $\nnz(\Ah_{\mathrm{THB}})\le \nnz(\Ah_{\mathrm{HB}}) \eqsim  p^d (N+M)$.
	Interestingly, for 2D domains $\Cost_{T}$ has the same order in $p$ as assembling $\Ah_{\mathrm{HB}}$. This shows that using the truncated hierarchical B-spline basis is, up to a $p$-independent factor, equivalent to the standard hierarchical B-spline basis for 2D domains.
	This does not generalize to higher dimensional domains.
\end{remark}
	
\begin{example}[Multipatch domains]\label{example:multipatch}

In many practical problems, the computational domain $\Omega$ is not
diffeomorphic to a $d$-dimensional cube. In such cases, the domain
$\Omega$ is typically partitioned into
sub-domains $\Omega_\ell$, $\ell=1,\dots,L$, each parametrized by a map 
$\textbf G_\ell$ defined on $[0,1]^d$ as in
\[
\Omega =\bigcup_{\ell=1}^L \Omega_\ell =\bigcup_{\ell=1}^L \,\textbf G_\ell([0,1]^d).
\]
On each of these sub-domains, independent bases $\Phi_\ell$ are defined.
Assuming that we want to solve a second-order PDE, we are typically interested in a
$H^1$-conforming discretization which means that the basis functions have to be continuous
at the interfaces between the patches. This is typically enforced by identifying the
basis functions that coincide on the the interface.

We want to explain the application of sum factorization in this context using the
abstract framework introduced in the beginning of this section. 
We combine the functions $\textbf G_\ell$ to one function
defined on $\widehat \Omega := \{1,\ldots, L \} \times [0,1]^d$ as follows
\[
	\textbf G: \widehat \Omega\rightarrow \Omega, \qquad \textbf G (\ell,\textbf x) := \textbf G_\ell(\textbf x).
\]
The sum factorization sub-domains coincide with the parametric patches:
\[
	\PPart=\big\{\{\ell\}\times [0,1]^d: \ell= 1,\dots,L \big\}.
\]
On each patch, we assume a tensor-product basis $\widehat\Phi_{\ell}$ to be given.
Since we want a continuous discretization space, we define the global function space
$\widehat\Phi$ as the union of the local bases $\widehat\Phi_1,\ldots,\widehat\Phi_L$, where we
(repeatedly) identify functions $\phi\in\widehat\Phi_\ell$ and
$\varphi\in \widehat\Phi_\gamma$ if 
\[
	\left\{ \begin{aligned}
		& \Gamma_{\ell,\gamma}:=\Omega_\ell\cap\Omega_\gamma \ne \emptyset \\
		& \phi(\textbf G_{\ell}^{-1} \textbf x)|_{\Gamma_{\ell,\gamma}}
			=\varphi(\textbf G_{\gamma}^{-1} \textbf x)|_{\Gamma_{\ell,\gamma}} \ne 0 .
	\end{aligned}\right .
\]

For each sub-domain $\Omega_\ell$, the PDE is pulled back to the parameter domain $\{\ell\}\times [0,1]^d$
analogous to Example~\ref{example:second-order}. On each of these domains, assembling can be
performed independently.
The cost of assembling each $\Ah_{\PBox}$ is described by~\eqref{localized-N}, \eqref{localized-N2}
and~\eqref{localized-Q}, that of applying $\Ah_{\PBox}$ by~\eqref{matrix-less-N-local}, \eqref{matrix-less-N2-local} and~\eqref{matrix-less-Q-local}.

In all cases of practical interest, the repetition ratio $R$ is small. If the 
local bases are tensor-product B-spline bases defined over open knot vectors
such that there are $N_{\PBox,\delta}\ge 3$ basis functions
in each direction $\delta=1,\ldots, d$, we easily observe that
$\prod_{\delta=1}^d (N_{\PBox,\delta}-2)$ basis functions vanish on $\partial \PBox$.
Thus, we obtain
\[
	R = \frac{\sum_{\PBox\in \PPart} N_{\PBox} }{N}
	  \le \frac{\sum_{\PBox\in \PPart} \prod_{\delta=1}^d N_{\PBox,\delta}}
				{\sum_{\PBox\in \PPart} \prod_{\delta=1}^d (N_{\PBox,\delta}-2)}
	  \le 3^d \lesssim1.
\]
The extension to bilinear forms involving derivatives is completely straight forward.
\end{example}

\begin{remark}
	In real-world applications, the domains obtained by the parameterization
	are often additionally trimmed.
	Extending the proposed techniques to trimmed domains would be a challenge.
	A possible approach would be to compute the matrix using sum factorization and 
	ignoring trimming at first.
	Then to recompute the matrix elements corresponding to the basis
	functions whose support is trimmed using a different technique.
	In reasonable cases, one expects that the number of basis functions with trimmed support is small compared to the total number of basis functions.
	Integration over trimmed elements is still an active research topic, cf. \cite{trimmed-quad,untrimming}.
\end{remark}

\section{An algebraic description}\label{sec:5}

The unifying idea of sum factorization is that the map from the value of the
coefficients on the quadrature points to the system matrix is linear and can be represented as a tensor;
the overall algorithm can be seen as a tensor contraction. Since most people are more familiar
with Kronecker products, we explain the algorithm using them.
The two descriptions correspond to one another by reinterpreting a order-$d$ tensor as the diagonal of a matrix.

Let $\Quad$ be lexicographically ordered and define
\[
	Q_{\Phi}=[\phi_n(\textbf x)]_{n=1,\dots,N}^{\textbf x\in\Quad}, \qquad Q_{\Psi}=[\omega(\textbf x)\psi_m(\textbf x)]_{m=1,\dots,M}^{\textbf x\in\Quad}
\]
and $F$ to be the diagonal matrix containing $\mathcal F(\textbf x)$ for $\textbf x \in\Quad$.
Then
\begin{equation}\nonumber%\label{eq:kronecker-1}
\Ah= (Q_{\Psi})^\top F\, Q_{\Phi}.
\end{equation}
The factorization can be now expressed by decomposing $Q_{\Phi}$ and $Q_\Psi$ as Kronecker products:
\[
	\begin{aligned}
	&Q_\Phi = \Kprod_{\delta=1}^d \Pfact{Q_\Phi}\delta,\qquad \Pfact{Q_\Phi}\delta=[\phif\delta_n(x_\delta)]_{n=1,\dots,N_\delta}^{x_\delta\in\Quad_\delta}\\
	&Q_\Psi = \Kprod_{\delta=1}^d \Pfact{Q_\Psi}\delta,\qquad \Pfact{Q_\Psi}\delta=	[\psif\delta_m(x_\delta)\omega_\delta(x_\delta)]_{m=1,\dots,M_\delta}^	{x_\delta\in\Quad_\delta}.
	\end{aligned}
\]
Since for all matrices $C$ and $D$, the Kronecker product satisfies
\[
	C\Kprods D =(C\Kprods I)(I\Kprods D),
\]
we have
\begin{equation}\nonumber%\label{eq:kronecker-2}
\Ah= (I\Kprods\dots \Kprods \Pfact{Q_\Psi}d)^\top  \dots(\Pfact{Q_\Psi}1 \Kprods\dots\Kprods I )^\top F(\Pfact{Q_\Phi}1 \Kprods\dots\Kprods I )  \dots (I\Kprods\dots\Kprods  \Pfact{Q_\Phi}d).
\end{equation}
Note that
\[
	(\Pfact{Q_\Psi}1 \Kprods\dots\Kprods I )^\top F(\Pfact{Q_\Phi}1 \Kprods\dots\Kprods I )=\begin{pmatrix}\ddots& &\\& \Pfact\Ah1_{x_2,\dots,x_d} &\\ & &\ddots \end{pmatrix}_{x_i\in\Quad_i},
\]
where the right hand side is the block diagonal matrix with blocks $\Pfact\Ah1_{x_2,\dots,x_d}$ for $x_i\in\Quad_i$, $i=2,\dots,d$. The blocks are lexicographically ordered along the diagonal.
More generally, after $\delta$ multiplications, we obtain
\[
	(I\Kprods\dots\Pfact{Q_\Psi}\delta \dots\Kprods I )^\top \dots F \dots (I\Kprods\dots\Pfact{Q_\Phi}\delta \dots\Kprods I )=\begin{pmatrix}\ddots& &\\& \Pprod\Ah\delta_{x_{\delta+1},\dots,x_d} &\\ & &\ddots \end{pmatrix}_{x_i\in\Quad_i}.
\]
The improved performance can then be easily explained by the cost of a matrix-matrix multiplication
$C\;D$, which is bounded by $\nnz (C)\; \mbox{max-nnz-per-row}(D)$.
These correspond to the factors in \eqref{sum factorization-nnz}.

\section{Pre-assembling costs}\label{sec:6}

The overall costs presented in the previous sections only covers the costs of assembling $\Ah$ or computing $\textbf v:=\Ah\textbf u$.
The computational costs of computing $\Phi_\delta(\Quad_\delta)$, $\Psi_\delta(\Quad_\delta)$,  $\mathcal F(\Quad)$ and $u(\Quad)$ have been ignored, as common in the literature. However it is worth to have a closer look onto the corresponding computational costs.
Our aim is to show algorithms for which the computation of  $\Phi_\delta$, $\Psi_\delta$, $\mathcal F$ be performed at a cost that is inferior to the cost of sum factorization.
For simplicity, we focus on the global sum factorization approach, but the results hold also for the localized
version as they apply to each box separately.

\subsection{Evaluation of the functions in the generating sets \texorpdfstring{$\Phi_\delta$}{Phi} and \texorpdfstring{$\Psi_\delta$}{Psi}}\label{sec:6:1}

The evaluation of a B-spline function $\phi_n$ or $\psi_m$ cannot be done in constant time. As we
have a tensor-product structure, an efficient approach is to pre-compute the function values for the
corresponding univariate functions $\phif\delta_n$ or $\psif\delta_m$.

For a standard B-spline basis $\Phi$, this yields a cost $\Cost$ which is bounded as follows
\begin{equation}\label{eq:costs-basis}
	\Cost
		\lesssim  \sum_{\delta=1}^d p_\delta^2 \# \Quad_\delta.
\end{equation}
The above cost is negligible in all reasonable situations if $d\ge2$.

\subsection{Evaluation of \texorpdfstring{$u\in\SPAN \Phi$}{u in span Phi}}\label{sec:6:3}

First we remember that
$$
u(x) = \sum_{n=1}^N \textbf{u}_n \phi_n(x).
$$

For the evaluation of $u$ at the quadrature points $\Quad$, we can again exploit the tensor-product structure of $\Phi$ and $\Quad$. This yields the following algorithm:

\begin{algorithm}[H]
	\caption{Recursive function computation}
	\label{alg-compute}
	\begin{algorithmic}[1] % The number tells where the line numbering should start
	\Procedure{Eval}{$[ \Quad_\delta]_{\delta=1}^d, [\Phi_\delta]_{\delta=1}^d,  [\textbf{u}_n]_{n=1}^N$}
	\If {d=0}
		\State \textbf{return} $\textbf{u}_1$
	\EndIf
	\State $\Pprod u d \gets 0$
	  \ForAll{$n_d\in \{1,\dots, {N_d}\}$} %\Comment{assemble $(d-1)$-variate matrices}
		 \State $\Pprod u{d-1}_{n_d}  \gets $\textsc{Eval}$ ([ \Quad_\delta]_{\delta=1}^{d-1},
		 [\Phi_\delta]_{\delta=1}^{d-1}, \left[\textbf{u}_{n}\right]_{n: \pi_d(n) = n_d}  )$ \label{alg_line:rec}
		\ForAll{ $x_d\in \Quad_d:\phif d_{n_d}(x_d)\ne 0$}
		\State $ \Pprod u d({\textbf x},x_d) \gets \Pprod u d({\textbf x},x_d)
						+ \phif d_{n_d}(x_d)\Pprod u{d-1}_{n_d} ({\textbf x})$\label{alg_line:sum}
 		\EndFor
	\EndFor
	\State \textbf{return} $\Pprod u d$
	\EndProcedure
	\end{algorithmic}
\end{algorithm}
Again, we derive the number of floating point operations, assuming that the functions in the
generating set $\Phi_\delta$ have already been evaluated.
By counting the number of invocations of lines~\ref{alg_line:rec} and~\ref{alg_line:sum}, using
$$
\sum_{n=1}^{N_\delta}\#\{x\in\Quad_\delta:\phif{\delta}_{n}(x)\ne 0\}
=\sum_{x\in\Quad_\delta} \#\{\phif \delta_n\in\Phi_\delta:\phif \delta_{n}(x)\ne 0\}\le p_\delta \#\Quad_\delta
$$
we obtain that the cost $\Pprod \Cost d$ satisfies
$$
		\Pprod \Cost d
			\lesssim N_d  \Pprod \Cost {d-1} + p_d\left( \prod_{\delta=1}^d \#\Quad_\delta \right).
$$
Recursively plugging this bound into itself yields
$$
		\Pprod \Cost d
			\lesssim \sum_{i=1}^{d} \left( \prod_{\delta=i+1}^d N_\delta \right)
			p_i \left( \prod_{\delta=1}^i \#\Quad_\delta \right)  
			\lesssim\left( \sum_{i=1}^{d}  p_i \right)
			\left( \prod_{\delta=1}^d \max\{ N_\delta ,\# \Quad_\delta\}\right) .
$$
Depending on the ratio between $\#\Quad_\delta$ and $N_\delta$, we obtain the following costs 
\begin{align}
\label{evaluation-N-u}
&\forall \delta,\,\#\Quad_\delta\lesssim  p_\delta N_\delta   &\Rightarrow&&& \Pprod\Cost d
\lesssim  p^{d+1} N,\\
\label{evaluation-N2-u}
&\forall \delta,\,\#\Quad_\delta\lesssim   N_\delta    &\Rightarrow&&& \Pprod\Cost{d}\lesssim  p \, N,\\
\label{evaluation-Q-u}
&\forall \delta,\,\#\Quad_\delta\gtrsim   N_\delta &\Rightarrow&&& \Pprod\Cost{d}\lesssim  p \,    \#\Quad .
\end{align}

\subsection{Evaluation of the coefficient function \texorpdfstring{$\mathcal F$}{ F} in IGA}\label{sec:6:2}

As mentioned in Section~\ref{sec:2}, we have to evaluate the function $\mathcal F$ for every quadrature point.
In classical Isogeometric Analysis, we assume that the computational domain $\Omega$ is parametrized by a diffeomorphism
$$
	\textbf{G}:\widehat{\Omega}:=[0,1]^d \rightarrow \Omega = \textbf{G}(\widehat{\Omega}) \subset \mathbb R^s,
$$
which is an element of $\SPAN \Phi$ and has the form
$$
\textbf G = \sum_{n=1}^N \textbf{c}_n \phi_n.
$$
Computing $\mathcal F$ involves the evaluation of $\textbf G$ and/or its derivatives, cf.~\eqref{eq:cdr:1}.
Each component of $\textbf G$, or of a derivative of $\textbf G$, is a spline function and can be evaluated by the algorithm in Section~\ref{sec:6:3}.
Consequently, the cost of evaluating $\textbf G$ and the required derivatives is as in \eqref{evaluation-N-u}, \eqref{evaluation-N2-u} and~\eqref{evaluation-Q-u}, where the number of components is independent of $p$, but it depends on $d$, $s$ and the PDE, and it is hidden in $\lesssim$.

The costs of computing derived quantities, such as the pseudo-inverse of the Jacobian matrix or its determinant, are then proportional to $\#\Quad$ with a rate that depends on $d$ and $s$, but that is independent of $p$.
In any case, the cost of computing derived quantities is dominated, for large $p$, by the cost of computing $\textbf G$ and the required derivatives.

\section{Numerical experiments}\label{sec:7}

We have implemented the global strategy, the per-macro-element strategy
and the per-element strategy
and have tested their behavior for a few sample problems.
Our implementation is a C++ code which has been carefully optimized. It is
available online\footnote{\url{https://github.com/IgASF/IgASF}} as a stand-alone
assembling library.

\begin{figure}[th]
	\begin{minipage}{.45\textwidth}\centering
	\includegraphics[height=.75\textwidth]{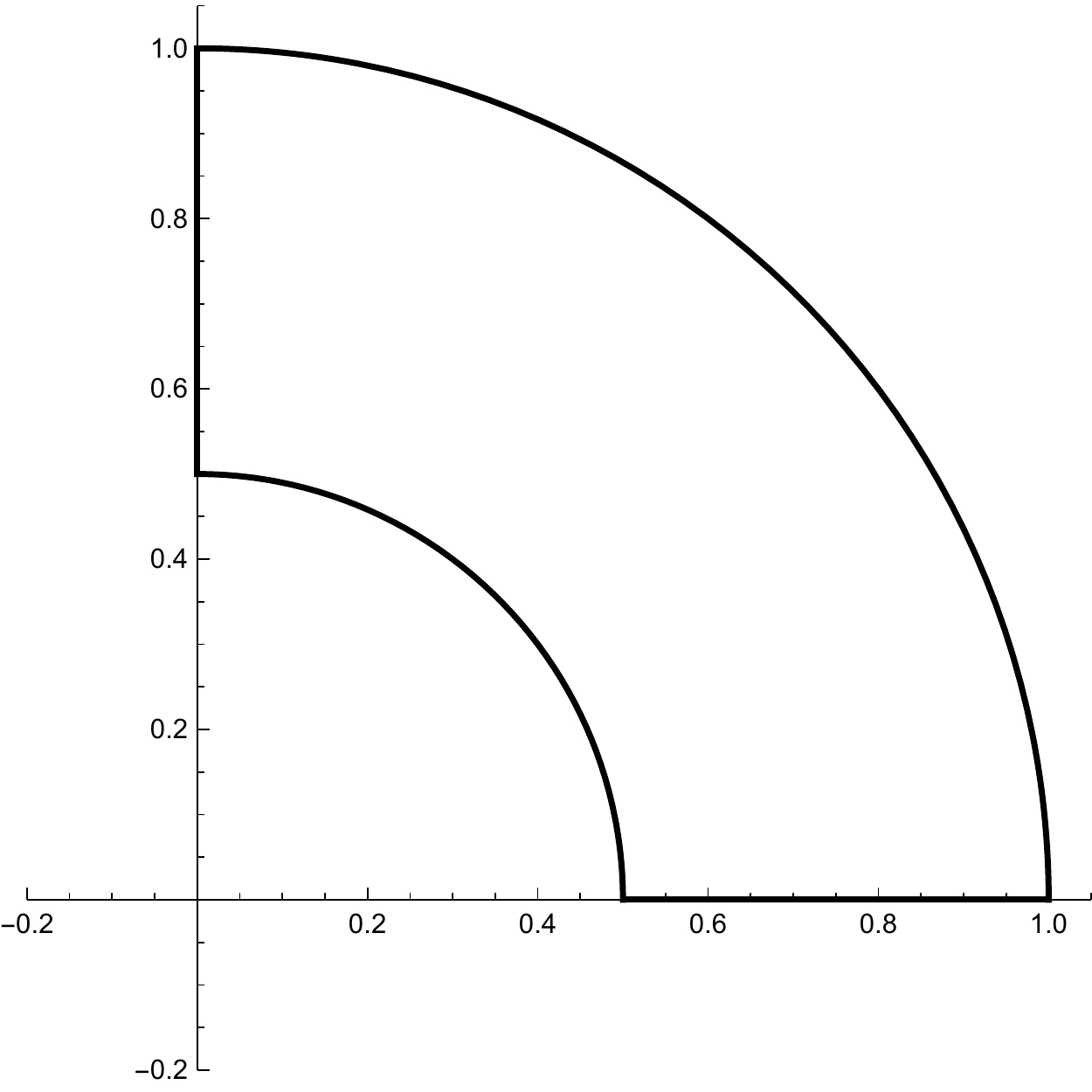}\\
	\scriptsize (a) 2D domain: quarter annulus
	\end{minipage}\hfill
	\begin{minipage}{.45\textwidth}\centering
	\includegraphics[height=.75\textwidth]{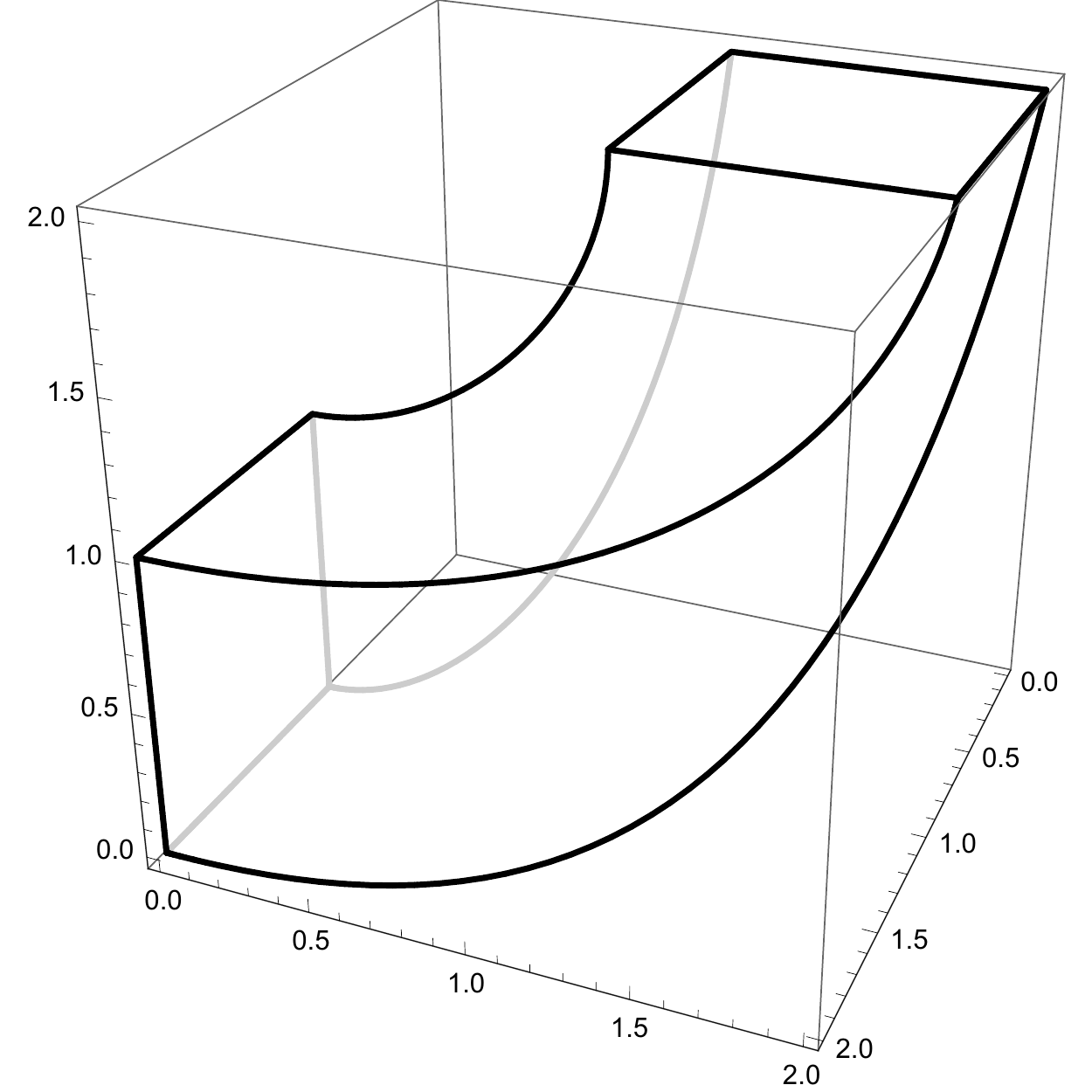}\\
	\scriptsize (b) 3D domain: bent and twisted box
	\end{minipage}
	\caption{\label{fig:domains} Computational domains}
\end{figure}

In the numerical experiments, we assemble a standard stiffness matrix for the 2D
and 3D domains depicted in Fig.~\ref{fig:domains}.
The 2D domain is decomposed into $200\times 200$ elements, the 3D domain into
$22\times 22 \times 22$ elements.
On each domain, we assemble the stiffness matrix for splines of several orders $p$
using Gauss quadrature of the same order and we compare the time used by different
algorithms. This was done on a single socket machine with an Intel(R) Core(TM)
i3-8100 CPU running at 3.60GHz.

In Fig.~\ref{fig:2d-times} and \ref{fig:3d-times}, we report the assembling times for
standard per-element assembling (\texttt{standard}), per-element sum factorization
(\texttt{element}), per-macro-element sum factorization (\texttt{macroS}) and global
sum factorization (\texttt{global}).  The macro-elements have size $p\times p$
in $2D$ and $p\times p \times p$ in $3D$.
Standard per-element assembling uses the same Gauss quadrature and was performed
using the freely available IGA library G+Smo \cite{gismoweb}.

In Fig.~\ref{fig:2d-times-macro} and \ref{fig:3d-times-macro}, we focus on the effect
of the macro-element size.
The considered approaches are: standard macro-elements (\texttt{macroS}) corresponding to a size of $p\times p$ in 2D and $p\times p\times p$ in 3D,
narrow macro-elements (\texttt{macroN}) having size $p\times 1$ in 2D and $p \times p \times 1$ in 3D and
rotated narrow macro-elements (\texttt{macroR}) having size $1\times p$ in 2D and $1 \times p \times p$ in 3D.
According to Example~\ref{example:narrow}, the narrow dimension should come last in sum factorization.
As our code does not reorder the dimensions we expect the same behavior for \texttt{macroS} and \texttt{macroN}, but a cost higher by a factor $p$ for \texttt{macroR}.
This means that \texttt{macroR} behaves as \texttt{element} in $2D$ and as $p^6 N$ in 3D.

\begin{figure}[th]\centering
\begin{tikzpicture}[scale=.75, transform shape]
	\begin{semilogyaxis}[
	minor x tick num=0,
	minor y tick num=9,
	xmin=1, xmax=22,
	ymin=0.05, ymax=1e4,
	y=.5cm,
	x=.5cm,
	grid=major,
	legend style={
	at={(0.75,0.05)},
	anchor=south
	}
	]

	\addplot+ [blue,mark=*] coordinates {
	(2,   2.4e+00)
	(3,   4.9e+00)
	(4,   8.5e+00)
	(5,   1.4e+01)
	(6,   2.1e+01)
	(7,   3.2e+01)
	(8,   4.7e+01)
	(9,   6.7e+01)
	(10,  9.4e+01)
	(11,  1.4e+02)
	(12,  1.9e+02)
	(13,  2.8e+02)
	(14,  3.8e+02)
	(15,  5.4e+02)
	(16,  7.8e+02)
	(17,  1.0e+03)
	(18,  1.3e+03)
	(19,  1.8e+03)
	(20,  2.4e+03)
	};%gismo 2D
	\addplot plot [domain=2:21,dashed,blue,mark=none,forget plot] { 1.64e-08*x^4.98*(200+x-1)^2 };

	\addplot [red,mark=diamond] coordinates {
	(2,   0.7)
	(3,   0.94)
	(4,   1.12)
	(5,   1.92)
	(6,   2.66)
	(7,   5.09)
	(8,   7.42)
	(9,   12.38)
	(10,  15.82)
	(11,  30.28)
	(12,  36.81)
	(13,  54.26)
	(14,  63)
	(15,  94.77)
	(16,  110.08)
	(17,  159)
	(18,  185.02)
	(19,  258.67)
	(20,  287.78)
	};%element 2D
	\addplot plot [domain=2:21,red,dashed,mark=none,forget plot]{ 4.91e-08*x^3.92*(200+x-1)^2 };
	
	\addplot [green,mark=triangle] coordinates {
	(2,   0.28)
	(3,   0.26)
	(4,   0.39)
	(5,   0.64)
	(6,   0.97)
	(7,   1.63)
	(8,   2.3)
	(9,   3.37)
	(10,  4.63)
	(11,  10.68)
	(12,  13.5)
	(13,  17.54)
	(14,  21.34)
	(15,  27.78)
	(16,  33.43)
	(17,  42.05)
	(18,  51)
	(19,  63.46)
	(20,  72.5)
	};% macroS 2D
	\addplot plot [domain=2:21,green,dashed,mark=none,forget plot]{ 7.94e-08*x^3.3*(200+x-1)^2 };
	
	\addplot [purple,mark=square] coordinates {
	(2,   0.07 )
	(3,   0.12 )
	(4,   0.21 )
	(5,   0.43 )
	(6,   0.67 )
	(7,   1.17 )
	(8,   1.65 )
	(9,   2.42 )
	(10,  3.42 )
	(11,  8.67 )
	(12,  10.73 )
	(13,  14.2 )
	(14,  17.31 )
	(15,  22.49 )
	(16,  27.78 )
	(17,  35.28 )
	(18,  42.08 )
	(19,  52.16 )
	(20,  63.66 )
	};%global 2D
	\addplot plot [domain=2:21,purple,dashed,mark=none,forget plot]{4.02e-08*x^3.47*(200+x-1)^2 };

	\legend{
	\texttt{standard}\\
	\texttt{element}\\
	\texttt{macroS}\\
	\texttt{global}\\
	}
	\end{semilogyaxis}\end{tikzpicture}
	\caption{Time in seconds for assembling the stiffness matrix on the 2D domain depending on the polynomial order $p$.}\label{fig:2d-times}
\end{figure}
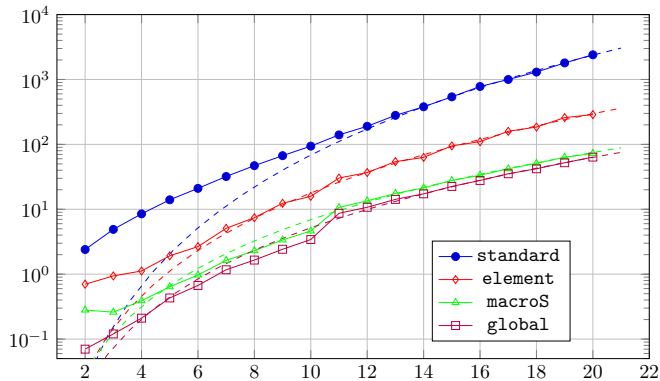

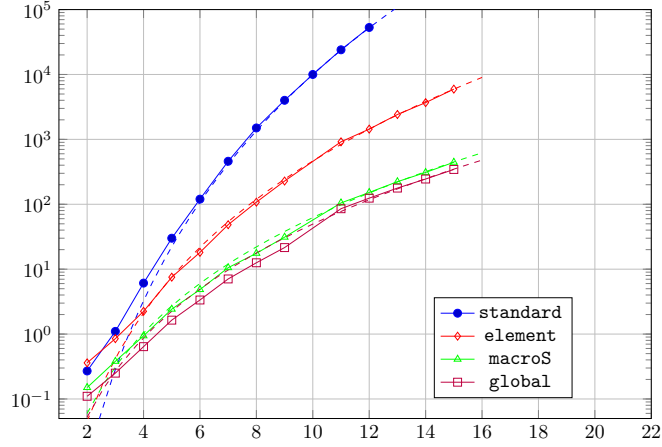
\begin{figure}[th]\centering
\begin{tikzpicture}[scale=.75, transform shape]
	\begin{semilogyaxis}[
	minor x tick num=0,
	minor y tick num=9,
	xmin=1, xmax=22,
	ymin=0.05, ymax=1e5,
	y=.5cm,
	x=.5cm,
	grid=major,
	legend style={
	at={(0.75,0.05)},
	anchor=south
	}
	]

	\addplot+ [blue,mark=*] coordinates {
	(2,   2.7e-01)
	(3,   1.1e+00)
	(4,   6.1e+00)
	(5,   3.0e+01)
	(6,   1.2e+02)
	(7,   4.6e+02)
	(8,   1.5e+03)
	(9,   4.0e+03)
	(10,  1.0e+04)
	(11,  2.4e+04)
	(12,  5.3e+04)
	};% gismo 3D
	\addplot plot [domain=2:13,dashed,blue,mark=none,forget plot] {2.99e-09*x^8.05*(22+x-1)^3 };

	\addplot [red,mark=diamond] coordinates {
	(2,    0.36)
	(3,    0.85)
	(4,    2.24)
	(5,    7.53)
	(6,    18.15)
	(7,    48.48)
	(8,    107.74)
	(9,    228.05)
	(10    417)
	(11,   916.94)
	(12,   1438.85)
	(13,   2416.95)
	(14,   3683.59)
	(15,   5944.21)
	};% element 3D
	\addplot plot [domain=2:16,red,dashed,mark=none,forget plot]{ 1.03e-7*x^5.18*(22+x-1)^3 };

	\addplot [green,mark=triangle] coordinates {
	(2,    0.15)
	(3,    0.38)
	(4,    0.94)
	(5,    2.42)
	(6,    4.85)
	(7,    10.54)
	(8,    17.57)
	(9,    31.13)
	(10    53.93)
	(11,   104.92)
	(12,   151.95)
	(13,   222.89)
	(14,   309.51)
	(15,   443.83)
	};% macroS 3D
	\addplot plot [domain=2:16,green,dashed,mark=none,forget plot]{ 3.64e-07*x^3.76*(22+x-1)^3 };
	
	\addplot [purple,mark=square] coordinates {
	(2,    0.11)
	(3,    0.25)
	(4,    0.64)
	(5,    1.64)
	(6,    3.35)
	(7,    7.07)
	(8,    12.55)
	(9,    21.51)
	(10    36.03)
	(11,   85.41)
	(12,   124.24)
	(13,   177.67)
	(14,   244.9)
	(15,   344.01)
	};% global 3D
	\addplot plot [domain=2:16,purple,dashed,mark=none,forget plot]{3.32e-07*x^3.70*(22+x-1)^3 };

	\legend{\texttt{standard}\\
	\texttt{element}\\
	\texttt{macroS}\\
	\texttt{global}\\
	}
	\end{semilogyaxis}\end{tikzpicture}
	\caption{Time in seconds for assembling the stiffness matrix on the 3D domain  depending on the polynomial order $p$.}\label{fig:3d-times}
\end{figure}

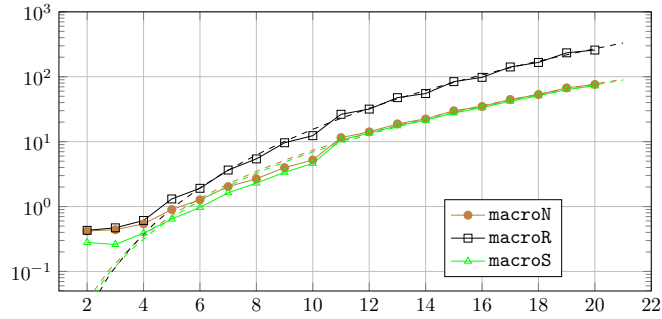
\begin{figure}[th]\centering
\begin{tikzpicture}[scale=.75, transform shape]
	\begin{semilogyaxis}[
	minor x tick num=0,
	minor y tick num=9,
	xmin=1, xmax=22,
	ymin=0.05, ymax=1e3,
	x=.5cm,
	y=.5cm,
	grid=major,
	legend style={
	at={(0.75,0.05)},
	anchor=south
	}
	]

	\addplot [brown,mark=*] coordinates {
	(2,   0.42)
	(3,   0.44)
	(4,   0.54)
	(5,   0.9)
	(6,   1.27)
	(7,   2.03)
	(8,   2.68)
	(9,   3.98)
	(10,  5.22)
	(11,  11.52)
	(12,  14.16)
	(13,  18.81)
	(14,  22.37)
	(15,  29.94)
	(16,  34.87)
	(17,  44.66)
	(18,  53.26)
	(19,  67.44)
	(20,  76.1)
	};% macroN 2D
	\addplot plot [domain=2:21,brown,dashed,mark=none,forget plot]{9.25e-08*x^3.26*(200+x-1)^2 };

	\addplot+ [black,mark=square] coordinates {
	(2,   0.43)
	(3,   0.47)
	(4,   0.61)
	(5,   1.31)
	(6,   1.91)
	(7,   3.65)
	(8,   5.42)
	(9,   9.62)
	(10,  12.36)
	(11,  26.39)
	(12,  31.75)
	(13,  47.64)
	(14,  55.12)
	(15,  84.13)
	(16,  97.83)
	(17,  141.52)
	(18,  166.1)
	(19,  233.71)
	(20,  257.81)
	};% macroR 2D
	\addplot plot [domain=2:21,dashed,black,mark=none,forget plot] { 3.63e-08*x^3.99*(200+x-1)^2 };

	\addplot [green,mark=triangle] coordinates {
	(2,   0.28)
	(3,   0.26)
	(4,   0.39)
	(5,   0.64)
	(6,   0.97)
	(7,   1.63)
	(8,   2.3)
	(9,   3.37)
	(10,  4.63)
	(11,  10.68)
	(12,  13.5)
	(13,  17.54)
	(14,  21.34)
	(15,  27.78)
	(16,  33.43)
	(17,  42.05)
	(18,  51)
	(19,  63.46)
	(20,  72.5)
	};% macroS 2D
	\addplot plot [domain=2:21,green,dashed,mark=none,forget plot]{ 7.94e-08*x^3.3*(200+x-1)^2 };

	\legend{
	\texttt{macroN}\\
	\texttt{macroR}\\
	\texttt{macroS}\\
	}
	\end{semilogyaxis}\end{tikzpicture}
	\caption{Time in seconds for assembling the stiffness matrix on the 2D domain with macro-elements  depending on the polynomial order $p$.}\label{fig:2d-times-macro}
\end{figure}

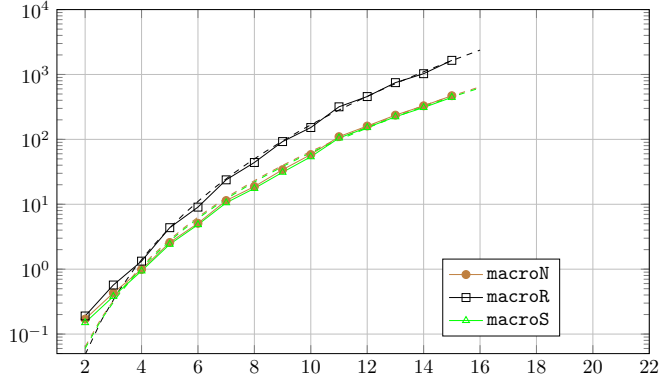
\begin{figure}[th]\centering
\begin{tikzpicture}[scale=.75, transform shape]
	\begin{semilogyaxis}[
	minor x tick num=0,
	minor y tick num=9,
	xmin=1, xmax=22,
	ymin=0.05, ymax=1e4,
	x=.5cm,
	y=.5cm,
	grid=major,
	legend style={
	at={(0.75,0.05)},
	anchor=south
	}
	]

	\addplot [brown,mark=*] coordinates {
	(2,    0.17)
	(3,    0.43)
	(4,    0.99)
	(5,    2.58)
	(6,    5.09)
	(7,    11.42)
	(8,    18.75)
	(9,    34.12)
	(10,   58.36)
	(11,   110.12)
	(12,   160.61)
	(13,   236.43)
	(14,   329.1)
	(15,   470.63)
	};% macroN 3D
	\addplot plot [domain=2:16,brown,dashed,mark=none,forget plot]{3.89e-07*x^3.75*(22+x-1)^3 };

	\addplot+ [black,mark=square] coordinates {
	(2,    0.19)
	(3,    0.57)
	(4,    1.33)
	(5,    4.36)
	(6,    9.04)
	(7,    23.8)
	(8,    44.03)
	(9,    92.49)
	(10,    152.16)
	(11,   317.57)
	(12,   456.64)
	(13,   748.58)
	(14,   1028.41)
	(15,   1650.32)
	};% macroR 3D
	\addplot plot [domain=2:16,dashed,black,mark=none,forget plot] { 1.74e-07*x^4.51*(22+x-1)^3 };

	\addplot [green,mark=triangle] coordinates {
	(2,    0.15)
	(3,    0.38)
	(4,    0.94)
	(5,    2.42)
	(6,    4.85)
	(7,    10.54)
	(8,    17.57)
	(9,    31.13)
	(10,    53.93)
	(11,   104.92)
	(12,   151.95)
	(13,   222.89)
	(14,   309.51)
	(15,   443.83)
	};% macroS 3D
	\addplot plot [domain=2:16,green,dashed,mark=none,forget plot]{ 3.64e-07*x^3.76*(22+x-1)^3 };
 
   \legend{
	\texttt{macroN}\\
	\texttt{macroR}\\
	\texttt{macroS}\\
	}
	\end{semilogyaxis}\end{tikzpicture}

	\caption{Time in seconds for assembling the stiffness matrix on the 3D domain using macro-elements  depending on the polynomial order $p$.}\label{fig:3d-times-macro}
\end{figure}

We observe that in any case, the sum factorization approaches are faster than
the standard approach. Conforming with the theory, we see that \texttt{global} is significantly
faster than \texttt{element}. Moreover, we obtain that the macro-element approaches
\texttt{macro} and \texttt{macroN} are indeed almost
as fast as the global approach \texttt{global}.

The results are used for fitting the parameters $c$ and $e$ in the formula
$$
t = c p^e N  (p)  ,
$$
where $t$ is the measured time, $N(p) = \prod_{\delta=1}^d(p+K_\delta-1)$ is the number of degrees of freedom and $K_\delta$ is the number of elements in the corresponding direction, i.e., $K_1=K_2=200$ for $d=2$ and $K_1=K_2=K_3=22$ for $d=3$.
The fitted curves are dashed in Fig.~\ref{fig:2d-times}--\ref{fig:3d-times-macro}.
The fitting procedure yields the following values for the exponent $e$:\\
\begin{center}
\begin{tabular}{lrrrr}
	\toprule
	&\multicolumn{2}{l}{2D domain}&\multicolumn{2}{l}{3D domain}\\
	\midrule
	&theory&experiments&theory&experiments\\
	\midrule
	\texttt{standard}  & 6 & 4.98 & 9 & 8.05\\
	\texttt{element}& 5 & 3.92 & 7 & 5.18\\
	\texttt{macroN} & 4 & 3.26 & 5 & 3.75\\
	\texttt{macroR} & 5 & 3.99 & 6 & 4.51\\
	\texttt{macroS} & 4 & 3.30 & 5 & 3.76\\
	\texttt{global} & 4 & 3.47 & 5 & 3.70\\
	\bottomrule
\end{tabular}\\[1.5em]
\end{center}

We observe that the exponents $e$  obtained in our experiments  are significantly lower than those predicted by the theory.
We believe that the huge difference between the speed in performing computation and the speed in accessing memory in modern processors, is masking one order in $p$.
For the 3D example, there is also another reason: the fitting is distorted by the small number of elements,
i.e., the approximation  $p^d N\approx \#\Quad$ does not apply:
$L_\delta=22$ so that $p N_\delta= 22 p+p^2 -p $ is twice bigger than $\#\Quad_\delta= 22 p$ for $p=20$.

In Fig.~\ref{fig:2d-times-parallel} and \ref{fig:3d-times-parallel}, we report the assembling times using a proof-of-concept multi-threaded approach based on macro-elements of size $p\times p$ in 2D and $p\times p\times p$ in 3D (\texttt{macroS}).
To have a sufficient number of macro-elements the 2D domain is split in $2000\times2000$ elements and the 3D domain is split in $64\times64\times64$ elements.
The lines correspond to the polynomial orders $p=4,6$ and $8$, and the abscissa is the number of parallel threads.
The tests have been executed on a dual socket machine with Intel(R) Xeon(R) CPU E5-2695 v4 running at 2.10GHz.
The figures show that the macro-element approach is viable to parallelization.
Table~\ref{tab:paralle-speedup} shows the corresponding speed-up factors, i.e. the ratio between the time for the single-threaded execution and the time for the multi-threaded execution.

\begin{table}
\begin{center}
\small
\begin{tabular}{ll|rrrrrrrrrrrr}
	\toprule
	 &  &\multicolumn{11}{c}{threads}\\
	$\dim$ &$p$&1&2&4&6&8&12&16&20&24&28&32&36\\
	\midrule
	2 &4 & 1& 1.9 & 3.5 & 4.7 & 5.9 & 7.6 & 9.3 & 10.6 & 11.3 & 12.4 & 13.5 & 12.8 \\
	2 &6 & 1& 2.0 & 3.6 & 4.9 & 6.2 & 8.5 & 10.3 & 11.9 & 14.0 & 15.2 & 15.8 & 16.1 \\
	2 &8 & 1& 2.0 & 3.8 & 5.1 & 6.6 & 8.9 & 11.0 & 13.2 & 15.2 & 16.9 & 18.4 & 19.8 \\
	\midrule
	3 &4 & 1& 1.9 & 3.7 & 5.2 & 6.6 & 8.9 & 11.0 & 12.9 & 14.2 & 16.7 & 15.0 & 17.4 \\
	3 &6 & 1& 2.0 & 3.7 & 5.2 & 6.6 & 9.0 & 10.9 & 12.8 & 14.3 & 15.2 & 15.6 & 17.6 \\
	3 &8 & 1& 2.0 & 3.7 & 5.0 & 6.6 & 8.3 & 10.9 & 11.4 & 13.6 & 13.6 & 17.4 & 17.5 \\
	\bottomrule
\end{tabular}
\end{center}
\caption{Speed-up factors for assembling the stiffness matrix using a parallel macro-element implementation.}\label{tab:paralle-speedup}
\end{table}

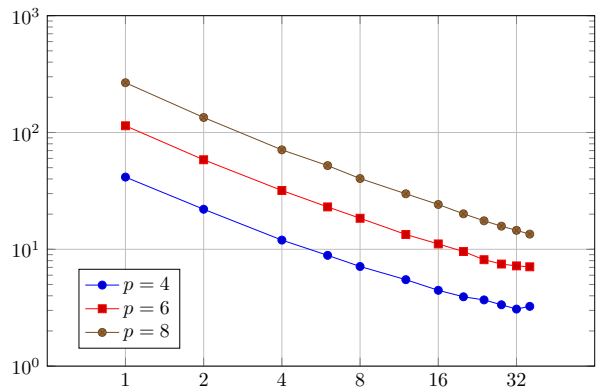
\begin{figure}[th]\centering
\begin{tikzpicture}[scale=.75, transform shape]
	\begin{loglogaxis}[
	xtick={1, 2, 4, 8, 16, 32},
	xticklabels={$1$, $2$, $4$, $8$, $16$, $32$},
	minor x tick num=0,
	minor y tick num=9,
	xmin=.5, xmax=64,
	ymin=1, ymax=1000,
	x=2cm,
	y=.9cm,
	grid=major,
	legend style={
	at={(0.15,0.05)},
	anchor=south
	}
	]
	\addplot coordinates {
		( 1, 41.5077)
		( 2, 22.0282)
		( 4, 11.9917)
		( 6, 8.88027)
		( 8, 7.12833)
		(12, 5.49706)
		(16, 4.45679)
		(20, 3.92254)
		(24, 3.6888)
		(28, 3.35139)
		(32, 3.07826)
		(36, 3.24829)
		};% deg 3 2D
	\addplot coordinates {
		( 1, 114.123)
		( 2, 58.4065)
		( 4, 31.8738)
		( 6, 23.0607)
		( 8, 18.4459)
		(12, 13.3925)
		(16, 11.1176)
		(20, 9.57081)
		(24, 8.14619)
		(28, 7.47907)
		(32, 7.21112)
		(36, 7.0827)
		};% deg 5 2D
	\addplot coordinates {
		( 1, 266.694)
		( 2, 134.424)
		( 4, 71.0182)
		( 6, 52.0786)
		( 8, 40.3489)
		(12, 29.8695)
		(16, 24.1967)
		(20, 20.135)
		(24, 17.5374)
		(28, 15.7785)
		(32, 14.5309)
		(36, 13.4914)
		};% deg 7 2D
	\legend{
		$p=4$\\
		$p=6$\\
		$p=8$\\
		}
	\end{loglogaxis}\end{tikzpicture}

	\caption{Times in seconds for assembling the stiffness matrix on the 2D domain using a parallel macro-element implementation.}\label{fig:2d-times-parallel}
\end{figure}

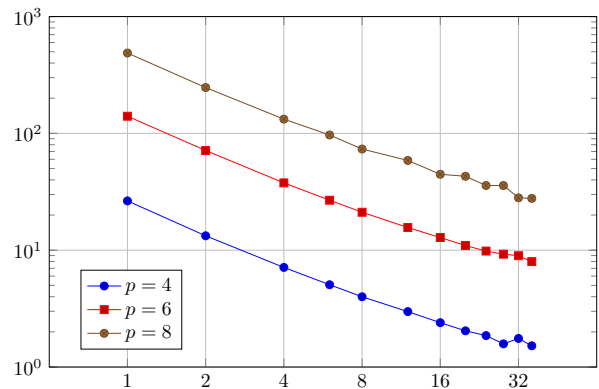
\begin{figure}[th]\centering
\begin{tikzpicture}[scale=.75, transform shape]
	\begin{loglogaxis}[
	xtick={1, 2, 4, 8, 16, 32},
	xticklabels={$1$, $2$, $4$, $8$, $16$, $32$},
	minor x tick num=0,
	minor y tick num=9,
	xmin=.5, xmax=64,
	ymin=1, ymax=1000,
	x=2cm,
	y=.9cm,
	grid=major,
	legend style={
	at={(0.15,0.05)},
	anchor=south
	}
	]
	\addplot coordinates {
		( 1, 26.4343)
		( 2, 13.2864)
		( 4, 7.12293)
		( 6, 5.0635)
		( 8, 3.99639)
		(12, 2.982)
		(16, 2.397)
		(20, 2.04168)
		(24, 1.85888)
		(28, 1.5788)
		(32, 1.75608)
		(36, 1.51988)
		};% deg 3 3D
	\addplot coordinates {
		( 1, 140.303)
		( 2, 71.4071)
		( 4, 37.7033)
		( 6, 26.7977)
		( 8, 21.1383)
		(12, 15.6453)
		(16, 12.8378)
		(20, 10.9742)
		(24, 9.83623)
		(28, 9.24794)
		(32, 8.9753)
		(36, 7.98465)
		};% deg 5 3D
	\addplot coordinates {
		( 1, 487.501)
		( 2, 246.831)
		( 4, 132.412)
		( 6, 96.9361)
		( 8, 73.4784)
		(12, 58.6642)
		(16, 44.6579)
		(20, 42.9512)
		(24, 35.8145)
		(28, 35.7794)
		(32, 28.0825)
		(36, 27.7732)
		};% deg 7 3D
   \legend{
	$p=4$\\
	$p=6$\\
	$p=8$\\
	}
	\end{loglogaxis}\end{tikzpicture}

	\caption{Times in seconds for assembling the stiffness matrix on the 3D domain using a parallel macro-element implementation.}\label{fig:3d-times-parallel}
\end{figure}

\section{Conclusions}\label{sec:8}

We have developed a unified complexity analysis for sum factorization approaches.
The theory shows
for several discretizations of interest in Isogeometric Analysis
that the computational costs can be reduced significantly by using sum factorization.
One of the advantages is that sum factorization can be applied with any tensor
product quadrature and that it yields, up to machine precision,
the same matrix as the standard approach.
In particular, significant savings are already achieved using standard Gauss quadrature.

We show that sum factorization does not yield its optimal complexity if it is applied on each element separately. However, one does not need to apply it globally to obtain its optimal complexity: it is sufficient to apply it to blocks of
at least $p$ elements in all directions, but one.

Moreover, we have shown that parallel implementations of localized sum factorization are a
viable strategy for fast assembling of the system matrix in IgA applications.

Additionally, we have examined the computational costs for matrix-free approaches.
For weighted quadrature, cf.~\cite{Sangalli:2017}, the
estimates~\eqref{matrix-less-N2} and \eqref{evaluation-N2-u} show
that the costs are $\eqsim p N$, which is
is a factor of $p^d$ less than the assembling of the matrix, cf.~\eqref{sum factorization-N}.
We see that Gauss quadrature does not allow an analogous speedup.
The estimates~\eqref{matrix-less-N} and \eqref{evaluation-N-u} show that, in this case,
the potential saving of a matrix-free approach corresponds only to a factor of~$p$.

\section*{Acknowledgments}

The first author has received funding from the European Research Council under the European Union's Seventh Framework Programme (FP7/2007-2013) / ERC grant agreement 339643.
The second author was partially funded by the Austrian Science Fund (FWF) under grants NFN S117-03
and P31048.
The authors want to thank the anonymous reviewers for their helpful comments and suggestions.

\section*{References}
\bibliographystyle{amsplain}
\bibliography{references}

\end{document}